\documentclass[12pt,a4paper,oneside]{amsart}%
\pdfoutput=1
\usepackage{amssymb}
\usepackage{esint}
\usepackage{mathtools}%[showonlyrefs]
\usepackage{hyperref}
\mathtoolsset{showonlyrefs}

\usepackage[english]{babel}
\usepackage[latin1]{inputenc}
\usepackage{graphics}
\theoremstyle{plain}
\newtheorem{theorem}{Theorem}
\newtheorem{lemma}[theorem]{Lemma}

\newtheorem{proposition}[theorem]{Proposition}
\theoremstyle{definition}
\newtheorem{definition}[theorem]{Definition}

\theoremstyle{remark}
\newtheorem{remark}[theorem]{Remark}
\newcommand{\R}{\mathbb R}
\newcommand{\C}{\mathbb C}
\newcommand{\N}{\mathbb N}
\newcommand{\Z}{\mathbb Z}

\newcommand{\h}{\mathcal H}

\newcommand{\brt}[1]{{\mathcal G_{#1}}}
\newcommand{\rt}[1]{{\mathcal I_{#1}}}

\newcommand{\der}{\mathrm{d}}
\newcommand{\eps}{\varepsilon}
\renewcommand{\phi}{\varphi}
\newcommand{\abs}[1]{\left| #1 \right|}

\newcommand{\sisus}{\operatorname{int}}
\newcommand{\spt}{\operatorname{spt}}

\newcommand{\tr}{\operatorname{tr}}
\renewcommand{\theta}{\vartheta}

\newcommand{\rfl}[1]{\widetilde{#1}}
\newcommand{\ext}[1]{\widehat{#1}}
\newcommand{\id}{\operatorname{id}}
\newcommand{\fun}[2]{{#2}^{#1}}
\newcommand{\lap}{\mathcal L}
\newcommand{\sff}[2]{\mathrm{I\!I}\!\left(#1,#2\right)}

\newcommand{\ip}[2]{\left\langle#1,#2\right\rangle}
\title{A reflection approach to the broken ray transform}
\author{Joonas Ilmavirta}
\address{Department of Mathematics and Statistics, University of Jyv\"askyl\"a, P.O. Box 35 (MaD) FI-40014 University of Jyv\"askyl\"a, Finland}
\email{joonas.ilmavirta@jyu.fi}
\date{\today}

\begin{document}
%\selectlanguage{finnish}

\begin{abstract}
We reduce the broken ray transform on some Riemannian manifolds (with corners) to the geodesic ray transform on another manifold, which is obtained from the original one by reflection.
We give examples of this idea and present injectivity results for the broken ray transform using corresponding earlier results for the geodesic ray transform.
Examples of manifolds where the broken ray transform is injective include Euclidean cones and parts of the spheres~$S^n$.
In addition, we introduce the periodic broken ray transform and use the reflection argument to produce examples of manifolds where it is injective.
We also give counterexamples to both periodic and nonperiodic cases.
The broken ray transform arises in Calder\'on's problem with partial data, and we give implications of our results for this application.
\end{abstract}



\maketitle

\section{Introduction}
\label{sec:intro}

Suppose we have an unknown compactly supported continuous function in the upper half plane $\{(x,y)\in\R^2;y\geq0\}$ and we know its integrals over all broken lines in the upper half plane, which reflect at $\R\times\{0\}$ according to the usual law of geometrical optics.
We can deduce the function from these integrals by reflecting the half plane to fill the entire plane and unfolding the broken rays into straight lines.
If we let $\rfl{f}(x,y)=f(x,\abs{y})$ for $(x,y)\in\R^2$, then we may reconstruct the integral of~$\rfl{f}$ over any straight line in the plane.
By injectivity of the Radon transform in the plane, we can deduce the original function~$f$ from this information.
In this article we generalize this reflection argument to show injectivity of the broken ray transform in various domains.

Let $(M,\partial M,g)$ be an $n$-dimensional compact Riemannian manifold with boundary.
We assume that the boundary $\partial M$ is a disjoint union of~$E$, $R$, and~$C$ such that~$E$ and~$R$ are open and $C=\partial E=\partial R$ in the topology of~$\partial M$.

We consider \emph{broken rays} to be piecewise geodesic paths~$\gamma$ on~$M$ such that
\begin{itemize}
	\item $\gamma$ starts and ends in the set $\bar{E}=E\cup C$,
	\item $\gamma$ is a geodesic in $\sisus M$, and
%	\item $\gamma$ intersects $\partial M$ finitely many times, and
	\item $\gamma$ is reflected on $R$ according to the usual reflection law: the angle of incidence equals the angle of reflection.
%$\gamma$ is reflected on $R$ as follows: Suppose a geodesic meets $R$ at $x\in R$. Let $\xi\in T_xM\setminus\{0\}$ be the direction of geodesic at $x$. There is a unique geodesic starting at $x$ in the reflected direction $\xi-2(\xi\cdot\nu)\nu$, where $\nu$ is the outer unit normal to the boundary at $x$. (If the geodesic is tangent to the boundary, no reflection thus occurs.) If this geodesic again meets $R$, we proceed similarly at the new point of intersection. The path $\gamma$ is obtained by concatenating such geodesics.
\end{itemize}
If convenient, we may also allow reflections on~$\bar{E}$; such paths can be constructed by concatenating a finite number of broken rays as defined above.
Since all broken rays have endpoints in the set~$E$, we call it the \emph{set of tomography}.

We ask the following questions:
If the integral of an unknown real valued function~$f$ on~$M$ is known over all broken rays, can~$f$ be reconstructed?
If yes, is the reconstruction stable?
How do answers to these questions depend on the regularity assumptions on~$f$, $g$ and~$M$?

To answer these questions, we reduce the problem to injectivity and regularity of the geodesic ray transform on Riemannian manifolds via reflections.
This can be done most naturally on manifolds with corners as discussed and proven in Section~\ref{sec:rfl}.
For manifolds with smooth boundary, more steps have to be taken, and they are outlined in Section~\ref{sec:spt-corner}.

%One could also try to adapt the old proofs for the geodesic ray transform to the case of broken rays.
%That may prove to be unnecessary if the present approach leads to good results.
%Besides answering the questions asked above we aim to construct a general scheme for converting broken ray problems to geodesic ray transform problems.
%This way a great variety of results could be obtained need for building a theory from scratch with broken rays.

We define the geodesic ray transform and the broken ray transform as follows:

\begin{definition}
\label{def:rt-brt}
For a manifold~$(M,g)$ with boundary we denote the set of all geodesics joining boundary points by~$\Gamma(M)$. For two classes of functions $F,H:M\to\R$ and any $h\in H$ we define the attenuated (geodesic) ray transform $\rt{h}:F\to\fun{\Gamma(M)}{\R}$ by
\begin{equation}
%\label{eq:}
\rt{h}f(\gamma) = \int_0^L f(\gamma(t))\exp\left(\int_0^t h(\gamma(s))\der s\right)\der t,
\end{equation}
when $\gamma:[0,L]\to M$ is a geodesic in $\Gamma(M)$ with unit speed.

For a set of tomography $E\subset\partial M$, we denote the set of broken rays from~$E$ to~$E$ by $\Gamma_E(M)$ (allowing reflections on~$E$).
We define similarly the attenuated broken ray transform $\rt{h}:F\to\fun{\Gamma_E(M)}{\R}$ by
\begin{equation}
%\label{eq:}
\brt{h}f(\gamma) = \int_0^L f(\gamma(t))\exp\left(\int_0^t h(\gamma(s))\der s\right)\der t.
\end{equation}

If attenuation nor the word `attenuated' is not mentioned, the attenuation is assumed to vanish identically.
\end{definition}

We also study the periodic broken ray transform where the entire boundary is reflecting and integrals of the unknown function are known over all periodic broken rays.
The precise definition is the following:

\begin{definition}
\label{def:pbrt}
Let~$M$ be a Riemannian manifold with boundary.
Let~$\Gamma$ be the set of all periodic broken rays in~$M$.
The mapping $\brt{}:C(M;\R)\to B(\Gamma,\R)$, $\brt{} f(\gamma)=\int_{\tr(\gamma)}f\der\h^1$, is the periodic broken ray transform.
\end{definition}

Using the reflection approach, we show that the broken ray transform is injective on the following manifolds (regularity requirements for functions vary, but $C_0^\infty$ suffices in each case):
\begin{itemize}
\item Euclidean domains where the reflecting part~$R$ of the boundary is part of a cone. This includes all polygons in the plane, where the reflecting part is at most two adjacent edges. Attenuation may also be included. (See Proposition~\ref{prop:ex-cone} and Remark~\ref{rmk:ex-att}.)
\item Quarter of the sphere~$S^n$ for $n\geq2$ where the set of tomography if half of the boundary. (See Proposition~\ref{prop:ex-thm}.)
\item The two dimensional hemisphere where the set of tomography is slightly larger than half of the boundary. (See Proposition~\ref{prop:ex-thm}.)
\item An octant of the sphere~$S^2$ for the periodic transform. (See Proposition~\ref{prop:ex-S/8}.)
\item The cube $[0,1]^n$, $n\geq2$, for the periodic transform. (See Proposition~\ref{prop:pbrt-cube}.)
\end{itemize}
More generic examples are given in Theorems~\ref{thm:brt2} and~\ref{thm:brt3}.
We also give the following counterexamples, for which the transform is not injective:
\begin{itemize}
\item Manifolds that contain a (generalized) reflecting tubular part. (See Proposition~\ref{prop:ctr-ex}.)
\item The disk for the periodic transform. (See Proposition~\ref{prop:ex-disk}.)
\end{itemize}

Eskin~\cite{eskin} reduced the recovery of an electromagnetic potential from partial data to injectivity of the broken ray transform, and showed the transform to be injective in a Euclidean domain with convex reflecting obstacles.
The broken ray transform has recently been studied in its own right~\cite{I:disk,H:square}.
This research has been motivated by the fact that Kenig and Salo~\cite{KS:calderon} reduced Calder\'on's problem with partial data to the injectivity of the broken ray transform.
We will discuss this in more detail in Section~\ref{sec:calderon} below.

Isakov~\cite{I:refl-calderon} used a reflection argument similar to ours for Calder\'on's problem directly.
Such arguments also appear in the study of billiards (see eg.~\cite{book-billiards}).

The recovery of a function from its ray transform is a well understood problem in a Euclidean domain (see textbooks~\cite{book-helgason,book-natterer}), and there are also a number of results on Riemannian manifolds (of which we mention~\cite{B:hyperbolic,SU:surface,UV:local-x-ray,K:spt-thm}) and also in greater generality (see e.g.~\cite{FSU:general-ray,AD:surface}).
The broken ray transform, however, is much less studied, which makes it appealing to try to reduce broken ray problems to the usual ray tomography.

It should be noted that while the geodesic ray transform is a good model for measuring the attenuation coefficient in a material with light, the broken ray transform is not a very good model if the light ray is allowed to reflect.
After a few reflections the signal is essentially lost, and reconstruction methods using broken rays with one reflection only are more appropriate for this application.
The model with one reflection (in the interior of the domain) is known as the V-line Radon transform~\cite{MNTZ:radon,A:radon}.
In the case of multiple reflections of light it is more appropriate to use the radiative transfer equation to model propagation of light.

This paper is organized as follows.
In Section~\ref{sec:calderon} we recall the relation between Calder\'on's problem and the broken ray transform and show how to translate the results in this paper to results for Calder\'on's problem.
Section~\ref{sec:ex1} gives examples of the reflection construction by proving injectivity of the broken ray transform in Euclidean cones.
More general examples on manifolds are given in Section~\ref{sec:appl}.
To prove these more general examples, reflected manifolds are constructed and studied in Section~\ref{sec:rfl}, and a generalization of the result for Euclidean cones is given in Theorem~\ref{thm:rfl}.
%Using Theorem~\ref{thm:rfl}, we prove injectivity of the broken ray transform for a fairly general class of manifolds.
The examples in Section~\ref{sec:appl} are based on this theorem.
The results up to this point require that the manifold~$M$ has a corner at~$C$.
A method for removing this restriction is presented in Section~\ref{sec:spt-corner}.
In Section~\ref{sec:ex} we give examples (and counterexamples) of specific manifolds where the broken ray transform is injective.
In Section~\ref{sec:periodic} we demonstrate by example that a similar reflection approach can also be used for the periodic broken ray transform, where the entire boundary is reflecting and one integrates over periodic broken rays.

\subsection{Relation to Calder\'on's problem}
\label{sec:calderon}

As mentioned above, our main motivation for the study of the broken ray transform comes from the Calder\'on problem with partial data.
The recent result by Kenig and Salo~\cite[Theorem~2.4]{KS:calderon} states roughly the following if the broken ray transform on~$M$ with set of tomography $E\subset\partial M$ is injective:
the partial Cauchy data for the Schr\"odinger equation on a manifold~$N$ determines the potential uniquely, provided that~$N$ contains a tubular part $[0,L]\times M$ and the inaccessible part of the boundary of~$N$ is contained in $[0,L]\times(\partial M\setminus E)$.
The Calder\'on problem can be reduced to the corresponding problem for the Schr\"odinger equation. %, if the conductivity and its normal derivative are known at the accessible part of the boundary. %Boundary determination results imply that these are known.
For basic results for the Calder\'on problem we refer to the review article~\cite{U:eit-calderon} and references therein.

As an example, we state the result for Calder\'on's problem arising from injectivity of the broken ray proven in Proposition~\ref{prop:ex-cone}.

\begin{figure}%
\includegraphics{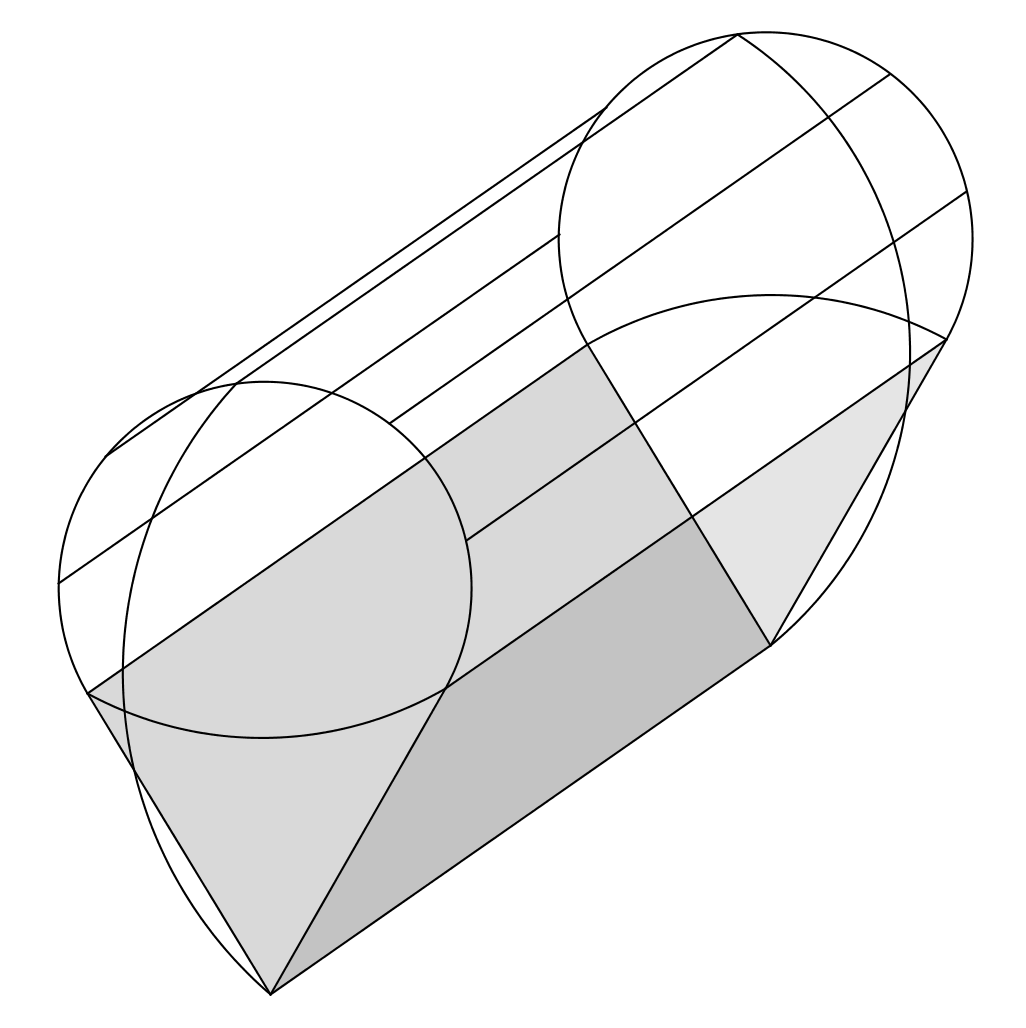}
%\psset{xunit=1.0cm,yunit=1.0cm,algebraic=true,dotstyle=o,dotsize=3pt 0,linewidth=0.8pt,arrowsize=3pt 2,arrowinset=0.25}
%\begin{pspicture*}(1.15,-3.44)(9.84,5.27)
%\pspolygon[linewidth=0pt,fillcolor=black,fillstyle=solid,opacity=0.1](4.92,-0.56)(3.44,-3.15)(7.67,-0.19)(9.16,2.4)
%\pspolygon[linewidth=0pt,fillcolor=black,fillstyle=solid,opacity=0.15](1.89,-0.6)(3.44,-3.15)(7.67,-0.19)(6.12,2.36)
%\parametricplot{-0.5074559731295754}{3.6749301870946427}{1*1.75*cos(t)+0*1.75*sin(t)+3.39|0*1.75*cos(t)+1*1.75*sin(t)+0.29}
%\parametricplot{-0.5074559731295745}{3.674930187094643}{1*1.75*cos(t)+0*1.75*sin(t)+7.63|0*1.75*cos(t)+1*1.75*sin(t)+3.25}
%\psline(3.44,-3.15)(7.67,-0.19)
%\psline(3.44,-3.15)(1.89,-0.6)
%\psline(3.44,-3.15)(4.92,-0.56)
%\psline(4.92,-0.56)(9.16,2.4)
%\psline(9.16,2.4)(7.67,-0.19)
%\psline(6.12,2.36)(1.89,-0.6)
%\psline(6.12,2.36)(7.67,-0.19)
%\parametricplot{2.3933830753876366}{4.000327877691593}{1*3.6*cos(t)+0*3.6*sin(t)+5.79|0*3.6*cos(t)+1*3.6*sin(t)+-0.43}
%\parametricplot{4.213442368195981}{5.2372171529486735}{1*3.1*cos(t)+0*3.1*sin(t)+3.37|0*3.1*cos(t)+1*3.1*sin(t)+2.12}
%\parametricplot{1.0718497146061927}{2.0956244993588764}{1*3.1*cos(t)+0*3.1*sin(t)+7.67|0*3.1*cos(t)+1*3.1*sin(t)+-0.32}
%\parametricplot{-0.8853183245125509}{0.9958439704121936}{1*3.2*cos(t)+0*3.2*sin(t)+5.64|0*3.2*cos(t)+1*3.2*sin(t)+2.29}
%\psline(1.64,0.33)(5.88,3.28)
%\psline(5.09,0.69)(9.33,3.65)
%\psline(4.44,1.69)(8.68,4.64)
%\psline(2.04,1.4)(6.28,4.36)
%\psline(3.15,2.02)(7.39,4.98)
%\end{pspicture*}
\caption{A domain in $\R^3$, which consists of a tubular part and is closed with $C^1$ caps in both ends. The partial Cauchy data determines conductivity in such a domain, where the inaccessible part is composed of the two highlighted plates. See Theorem~\ref{thm:ex-calderon} for details.}%
\label{fig:calderon}%
\end{figure}

\begin{theorem}
\label{thm:ex-calderon}
Let $a\in\R$ be any constant, and define the gutter $G_a=\R\times C_a$, where $C_a=\{(x_2,x_3)\in\R^2;x_3>a\abs{x_2}\}$.
Let $\Omega\subset G_a$ be a bounded domain such that for some $L>0$ we have $\Omega\cap([0,L]\times\R^2)=[0,L]\times\Omega_0$.
%Suppose for some~$r>0$ the transversal domain $\Omega_0$ satisfies $B(0,r)\cap\partial\Omega_0\subset\partial C_a$, and that $\partial\Omega$ is~$C^1$ outside the line $[0,L]\times\{(0,0)\}$. %tube $[0,L]\times\partial\Omega_0$.
Suppose that $\partial\Omega$ is~$C^1$ outside the line $[0,L]\times\{(0,0)\}$. %tube $[0,L]\times\partial\Omega_0$.
(One such domain is sketched in Figure~\ref{fig:calderon} for~$a>0$.)

Let $\Gamma_i=[0,L]\times(\partial\Omega_0\cap\partial C_a)\subset\partial\Omega$ be the inaccessible part of the boundary and denote the accessible part by $\Gamma_a=\partial\Omega\setminus\Gamma_i$.
Then the partial Cauchy data
%$\{(u|_{\Gamma_a},\partial_\nu u|_{\Gamma_a});(-\Delta+q)u=0,u|_{\Gamma_i}=0,u\in L^2(\Omega),\Delta u\in L^2(\Omega)\}$
\begin{equation}
%\label{eq:}
\{(u|_{\Gamma_a},\partial_\nu u|_{\Gamma_a});(-\Delta+q)u=0,u|_{\Gamma_i}=0,u\in L^2(\Omega),\Delta u\in L^2(\Omega)\}
\end{equation}
determines the potential $q\in C(\bar{\Omega})$ uniquely in~$\Omega$.
%Then partial Cauchy data for the Schr\"odinger operator $(-\Delta+q)$ determines the potential $q\in C(\bar{\Omega})$ uniquely in $\Omega$, if $[0,L]\times(\partial\Omega_0\cap\partial C_a)\subset\partial\Omega$ is inaccessible for measurements and has zero Dirichlet data, and both Dirichlet and Neumann data are measured on the rest of~$\partial\Omega$.
(The inaccessible part of the boundary is shadowed in Figure~\ref{fig:calderon}.)

This implies that the similarly partial Cauchy data for the Calder\'on problem determines a $C^2$ conductivity uniquely.
\end{theorem}

\begin{proof}
By the result~\cite[Theorem~2.4]{KS:calderon} it suffices to show that the broken ray transform is injective on the transversal manifold~$\Omega_0$ with all constant attenuations, where the set of tomography is $E=\partial\Omega_0\setminus\partial C_a$.
But this is done in Proposition~\ref{prop:ex-cone}(2) below; the parameters $a$ and $m$ are related by $a=\arctan(\frac{\pi}{2}(1-1/m))$.
Although we only prove Proposition~\ref{prop:ex-cone}(2) without attenuation, the same proof is valid for any constant attenuation as noted in Remark~\ref{rmk:ex-att}.
\end{proof}

Theorem~\ref{thm:ex-calderon} was proven for~$a=0$ (half space) by Isakov~\cite{I:refl-calderon} by another reflection method.
For~$a\leq0$ it was proven by Kenig and Salo~\cite{KS:calderon}; in this case it suffices to study broken rays without reflections.
Our result generalizes the previous ones to~$a>0$.
Injectivity results for the broken ray transform can be turned into partial data results for Calder\'on's problem in corresponding tubular domains; Theorem~\ref{thm:ex-calderon} is an example of this.

\begin{remark}
A partial version of Theorem~\ref{thm:ex-calderon} remains true if the corner of the transversal domain~$\Omega_0$ is smoothed out as follows.
Let~$\Omega_0$ be a transversal domain satisfying the assumptions of the theorem.
Suppose then~$\Omega_0'\subset\Omega_0$ is a subdomain such that for some~$\eps\in(0,r)$ we have $\Omega_0\setminus B(0,\eps)=\Omega_0'\setminus B(0,\eps)$.
Now $E=\partial\Omega_0\setminus\partial C_a\subset\partial\Omega_0'$.
As demonstrated at the end of Section~\ref{sec:rmk}, if the broken ray transform of~$f\in C(\overline{\Omega_0'})$ is known, one can determine~$f$ outside~$B(0,\eps)$.
Thus for a domain~$\Omega'\subset\R^3$ with transversal domain~$\Omega_0'$ as in the theorem the partial Cauchy data with $\Gamma_i=[0,L]\times(\partial\Omega_0'\setminus E)$ determines~$q$ outside the tube $[0,L]\times B(0,\eps)$.
The domain~$\Omega'$ can have a smooth boundary, unlike~$\Omega$.

Full recovery is possible if the smoothened tip of the cone is in not reflective but available for measurements; see Lemma~\ref{lma:subset} below.
\end{remark}

\section{First examples}
\label{sec:ex1}

We present some examples in the following proposition which demonstrate the idea that we wish to generalize.
We begin with a lemma that contains a general observation.
Here we denote by~$C_{pw}(\bar\Omega,\R)$ the piecewise continuous functions from~$\bar\Omega$ to~$\R$.

\begin{lemma}
\label{lma:subset}
Suppose the broken ray transform on a domain $\Omega\subset\R^n$ is injective on~$C_{pw}(\bar\Omega,\R)$ with some set of tomography~$E$.
Then it is also injective on any subdomain $\Omega'\subset\Omega$ (now on~$C_{pw}(\overline{\Omega'},\R)$) with a new set of tomography $E'=\partial\Omega'\setminus(\partial\Omega\setminus E)$.
\end{lemma}

The lemma is true also for manifolds and~$L^p$ functions and the proof is the same.

\begin{proof}[Proof of Lemma~\ref{lma:subset}]
Suppose $f\in C_{pw}(\overline{\Omega'},\R)$ integrates to zero over all broken rays.
Define then $g:\Omega\to\R$ by letting~$g=f$ on~$\Omega'$ and~$g=0$ on~$\Omega\setminus\Omega'$.
We clearly have $g\in C_{pw}(\bar\Omega,\R)$.

Take now any broken ray in~$\Omega$.
Intersecting it with~$\Omega'$ gives segments which are broken rays in~$\Omega'$ (endpoints on~$E'$).
Using this decomposition and the definition of~$g$, we observe that the broken ray transform of~$g$ vanishes.
By assumption this implies~$g=0$ and hence~$f=0$.
\end{proof}

\begin{proposition}
\label{prop:ex-cone}
The broken ray transform is injective in the following Euclidean domains~$\Omega$ for functions in $C_{pw}(\bar{\Omega},\R)$ (and thus also on $C(\bar{\Omega},\R)$):
\begin{enumerate}
\item A domain $\Omega\subset\R^2$ with $\partial\Omega\setminus E$ on a cone with opening angle $\pi/m$, $m\in\N$, and any set of tomography~$E$. For example in polar coordinates
\begin{equation}
\label{eq:ex1}
\Omega=\{(r,\theta):0<\theta<\pi/m,0<r<h(\theta)\}
\end{equation}
for continuous functions $h>0$, where $E=\{(h(\theta),\theta):0<\theta<\pi/m\}$.
\item The previous example works for all $m\geq1/2$ without the restriction $m\in\N$.
(That is, the opening angle may be anything in the range~$(0,2\pi]$.)
\item Also higher dimensional cones
\begin{equation}
%\label{eq:}
\Omega=\{(x',x)\in\R^{n-1}\times\R:\abs{x'}<kx,0<\abs{(x',x)}<h((x',x)/\abs{(x',x)})\}
\end{equation}
for continuous functions $h:S^{n-1}\to(0,\infty)$ and parameters $k\in(0,\infty)$ such that $\pi/2\arctan(k)\in\N$.
\item %The previous example works for all $k\in\R$ without the restriction $\pi/2\arctan(k)\in\N$.
The previous example works for all $k\in\R$ even without the restriction that $\pi/2\arctan(k)\in\N$.
(That is, the opening angle may be anything in the range~$(0,2\pi]$.)
\end{enumerate}
These injectivity results are also true for any subdomain $\Omega'\subset\Omega$ with new reflecting set $R'=R\cup\partial\Omega'$.
In particular, the cone need not contain a nonsmooth tip.
\end{proposition}
\begin{proof}
The last remark follows from Lemma~\ref{lma:subset}.

%(1) Reflect (do not rotate) a copy of the domain with respect to the horizontal axis (as should be indicated in a figure). The unknown function $f$ can be reflected as well by letting $f(x,y)=f(x,-y)$ when $y<0$ and $f(x,0)=0$. We then have a domain
%\begin{equation}%\label{eq:}
%\Omega'=\{(r,\theta):-\pi/m<\theta<\pi/m,0<r<f(\abs{\theta})\}.
%\end{equation}
%With suitable rotations (but no reflections) the domain $\Omega'$ can be similarly copied to fill...

(1)
%Reflect and rotate $m$ copies of the domain $\Omega$ and glue them along the reflecting part of~$\partial\Omega$.
Reflect~$\Omega$ over one side of the cone.
Then reflect this new copy of~$\Omega$ over the other side of the cone, and carry on until there are $2m$ copies of the original domain.
Because the opening angle is $\pi/m$, the total angle adds up to $2m\times\pi/m=2\pi$ and the copies of~$\Omega$ form a domain $\rfl{\Omega}$.
This reflection process is illustrated in Figs.~\ref{fig:cone1-pre} and~\ref{fig:cone1-post}; the domain $\Omega$ in Fig.~\ref{fig:cone1-pre} with $m=3$ is copied and reflected to constitute the domain $\rfl{\Omega}$ in Fig.~\ref{fig:cone1-post}.
Note that if one continues the construction by reflecting the last ($2m$th) copy of~$\Omega$, one ends up with the original domain in its original position.

Let $p:\bar{\rfl{\Omega}}\to\bar{\Omega}$ be the natural projection map that undoes the copying, rotating and reflecting done in the construction of~$\rfl{\Omega}$.
Let $f\in C_{pw}(\bar{\Omega},\R)$ be any function.
We define a reflected version of~$f$ by letting $\rfl{f}=f\circ p\in C_{pw}(\bar{\rfl{\Omega}},\R)$.

\begin{figure}%
\includegraphics{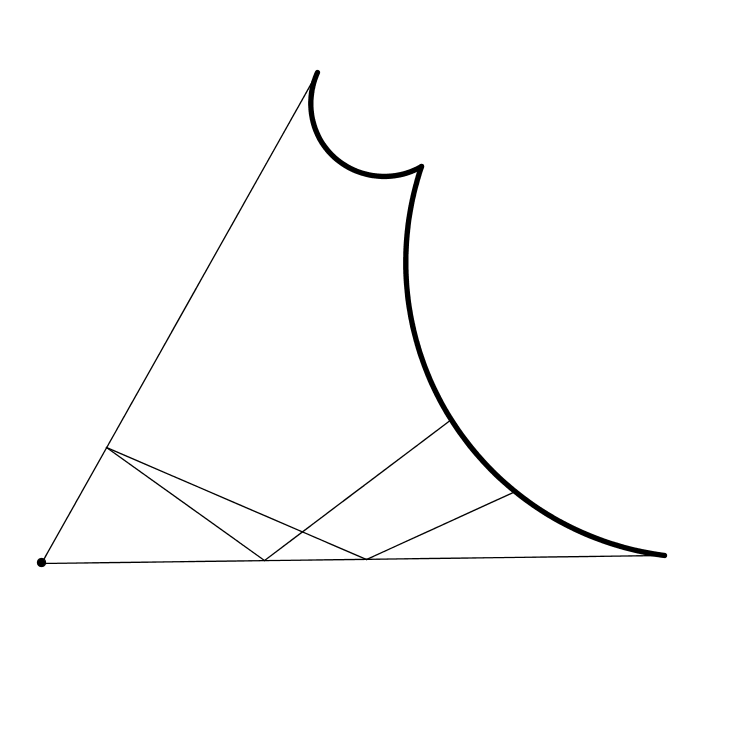}
\caption{A conical domain $\Omega$ with opening angle $\pi/3$ and a broken ray. See the proof of Proposition~\ref{prop:ex-cone}(1) for details.}%
\label{fig:cone1-pre}%
\end{figure}

\begin{figure}%
\includegraphics{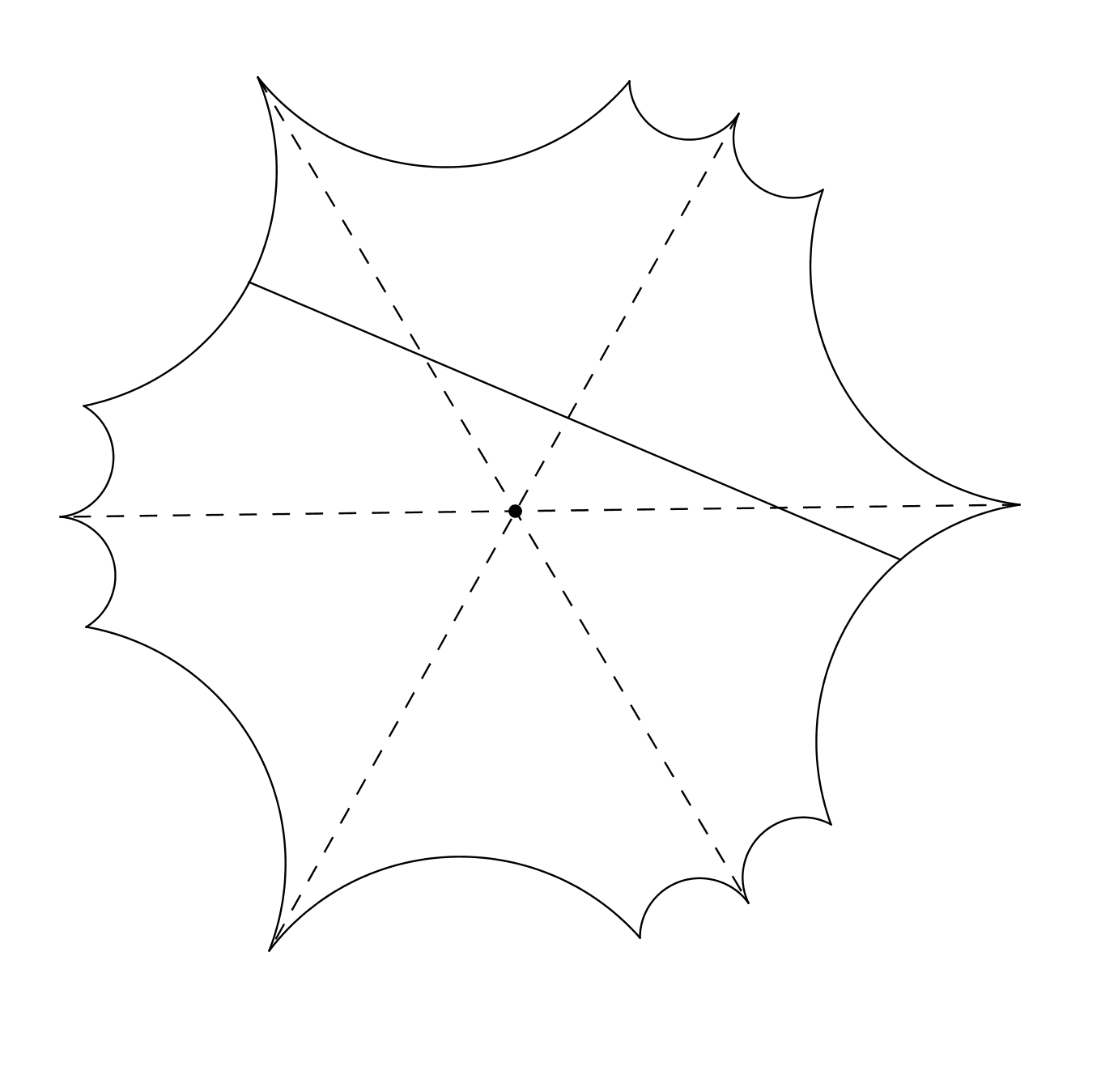}
\caption{A reflected domain $\rfl{\Omega}$ and a line. The line corresponds to the broken ray in Fig.~\ref{fig:cone1-pre}. See the proof of Proposition~\ref{prop:ex-cone}(1) for details.}%
\label{fig:cone1-post}%
\end{figure}

Take any line $\phi$ in $\rfl{\Omega}$ that does not meet the origin (for example the one in Fig.~\ref{fig:cone1-post}).
Let $\gamma=p\circ\phi$; then $\gamma$ is a broken ray in $\Omega$ (the broken ray in Fig.~\ref{fig:cone1-pre} corresponds to the line in Fig.~\ref{fig:cone1-post} in this way).
Because of this correspondence between $\gamma$ and~$\phi$ we write $\phi=\rfl{\gamma}$.
This correspondence of lines and broken rays is illustrated in Figs.~\ref{fig:cone1-pre} and~\ref{fig:cone1-post}.
(Note that for each broken ray~$\gamma$ in~$\Omega$ that does not hit the tip of the cone there are $2m$ lines~$\rfl{\gamma}$ in~$\rfl{\Omega}$.)

Since we have
\begin{equation}
\label{eq:rfl-int}
\int_{\rfl{\gamma}}\rfl{f}\der s
=
\int_{\gamma}{f}\der s,
\end{equation}
we may construct the Radon transform of~$\rfl{f}$ from the broken ray transform of~$f$.
In particular, vanishing broken ray transform of~$f$ in~$\Omega$ implies that the Radon transform of~$\rfl{f}$ vanishes.

Since the Radon transform is injective, vanishing broken ray transform of~$f$ implies that $\rfl{f}=0$ and so~$f=0$.

(2) Just like above, reflect and copy the domain in the plane.
The plane cannot be filled as nicely, but it is enough to construct a cone~$\rfl{\Omega}$ of opening angle at least~$\pi$.
Now~$\rfl{\Omega}$ does not cover all angles as in part~(1) above, but this problem can be bypassed.

Let $R\geq\max h$.
We cover the angle left out by~$\rfl{\Omega}$ by a compact cone~$C$ with radius~$R$ centered at the same point as copies of~$\Omega$ (the origin).
This construction is demonstrated in Figs.~\ref{fig:cone2-pre} and~\ref{fig:cone2-post} (analogously to Figs.~\ref{fig:cone1-pre} and~\ref{fig:cone1-post}).

%In polar coordinates we have
%\begin{equation}%\label{eq:}
%\rfl{\Omega}=\{(r,\theta):\theta\in(a,b),0<r<\rfl{h}(\theta)\}
%\end{equation}
%for some function $\rfl{h}:(a,b)\to(0,R)$. Here $(b-a)m/\pi\in\N$ is the number of copies of $\Omega$ used to construct $\rfl{\Omega}$ (at most $2m$).
%Let $C$ be the cone
%\begin{equation}%\label{eq:}
%C=\{(r,\theta):\theta\notin(a,b),0\leq r\leq R\}.
%\end{equation}
%Since the opening angle $b-a$ of $\rfl{\Omega}$ is greater than $\pi$, the cone $C$ is convex. It is also compact.

\begin{figure}%
\includegraphics{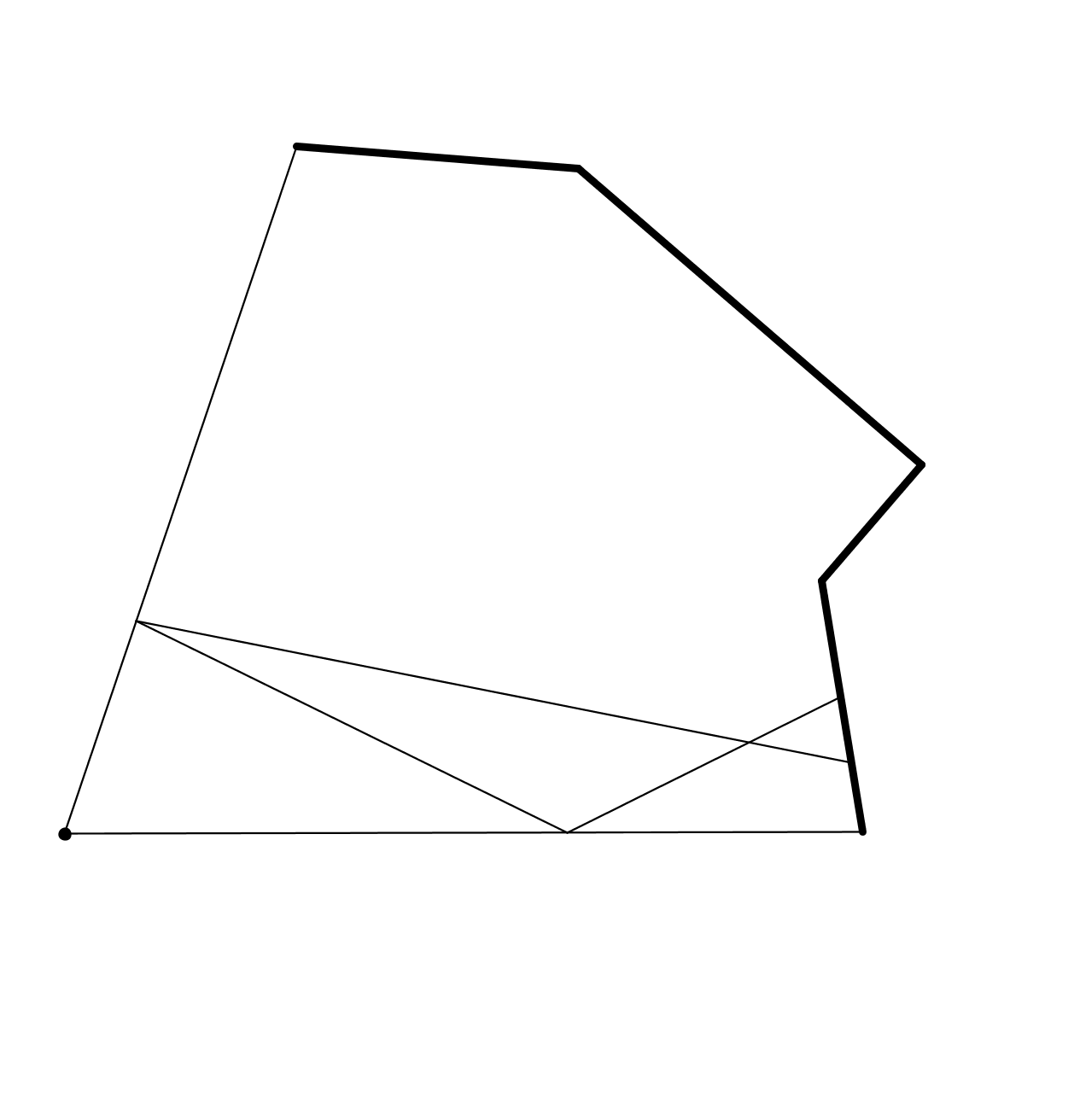}
\caption{A conical domain $\Omega$ with a broken ray. See the proof of Proposition~\ref{prop:ex-cone}(2) for details.}%
\label{fig:cone2-pre}%
\end{figure}

\begin{figure}%
\includegraphics[width=\columnwidth]{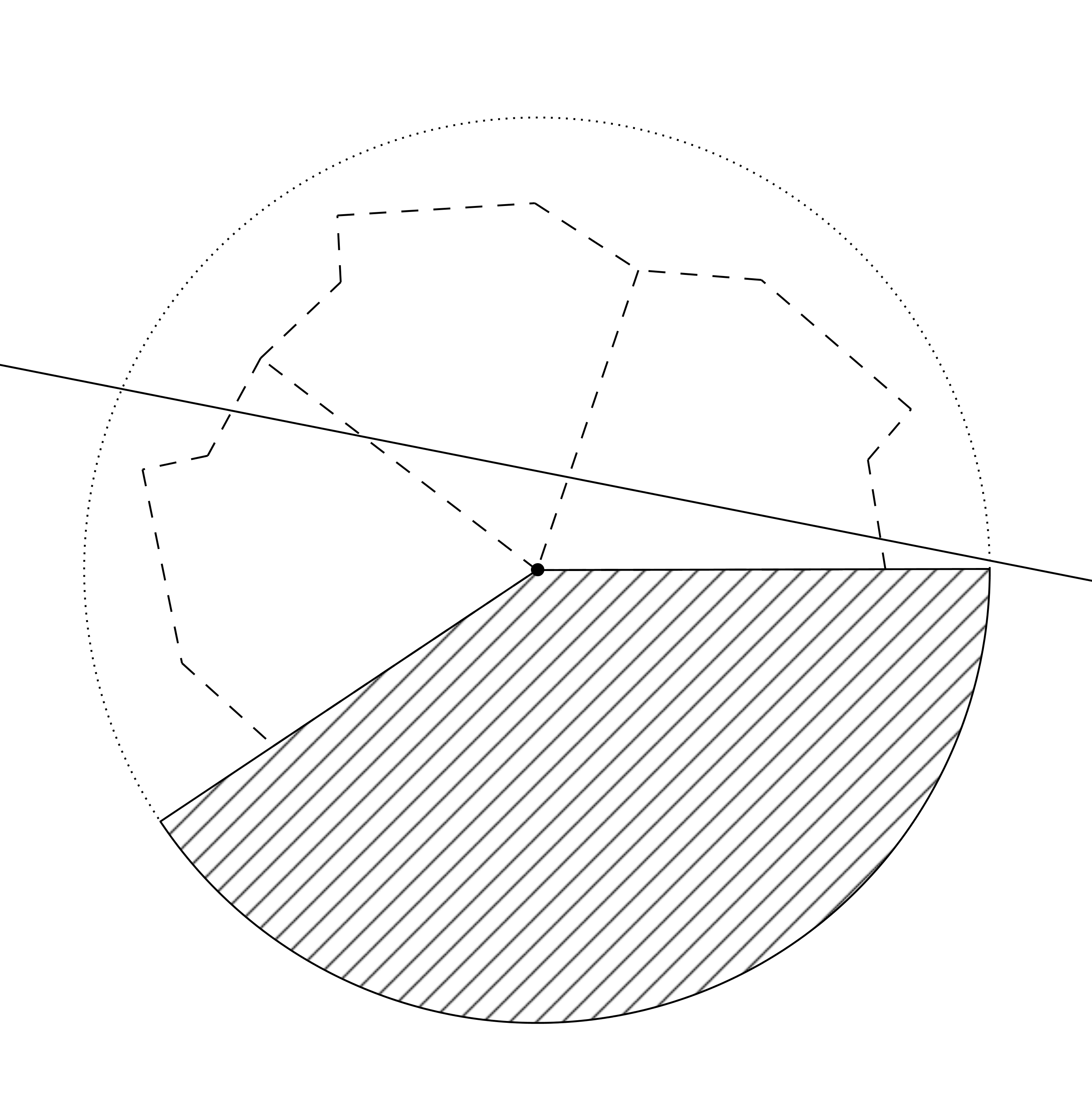}
\caption{The reflected domain $\rfl{\Omega}$ contains three copies of~$\Omega$. We fill in the angle with a cone. The line corresponds to the broken ray in Fig.~\ref{fig:cone2-pre}. See the proof of Proposition~\ref{prop:ex-cone}(2) for details.}%
\label{fig:cone2-post}%
\end{figure}

For $f:\bar{\Omega}\to\R$, define $\rfl{f}$ in $\rfl{\Omega}$ as above and extend by zero to~$\R^2$.
Suppose the broken ray transform of $f$ vanishes.

As in the proof of part~(1), the vanishing broken ray transform of~$f$ in~$\Omega$ implies that the integral of~$\rfl{f}$ over a line $L\subset\R^2$ is zero whenever~$L\cap C=\emptyset$.
(One such line~$L$ is drawn in Fig.~\ref{fig:cone2-post}.)
By Helgason's support theorem~\cite[Theorem~2.6]{book-helgason} (and a little mollification argument, see e.g.~\cite[proof of proposition~5]{I:disk}) $\rfl{f}$ vanishes outside~$C$ and thus especially in~$\rfl{\Omega}$.
Thus the original function~$f$ vanishes in~$\Omega$.

(3) The case $n=1$ is literally part~(1) and others can be reduced to it.

Let $P_2^n$ be the Grassmannian of all two dimensional subspaces of~$\R^n$.
For symmetry reasons every broken ray in the cone $\Omega$ is confined to some $E\in P_2^n$.
For any $E\in P_2^n$ part~(1) shows injectivity in the planar cone $\Omega\cup E$; $\pi/2\arctan(k)=m$ guarantees that the opening angle is $\pi/m$.
Since $\Omega=\bigcup_{E\in P_2^n}E\cap\Omega$, the broken ray transform is injective in~$\Omega$.

(4) Follows from part~(2) just like in part~(3) follows from part~(1).
\end{proof}

In the above proposition parts~(1) and~(3) follow trivially from~(2) and~(4), but we present them separately since their proofs are somewhat different.
In particular,~(2) and~(4) are based on Helgason's support theorem. %, but as simple generalizations to manifolds have not been proven.
%The support theorem requires compactness, whence it is not clear at all whether~(2) and~(4) admit an unbounded generalization such as~(5).

The broken ray transform is also injective on unbounded cones of the types~(1) and~(3) with $h\equiv\infty$, since the Radon transform is injective in the whole plane.
The set of tomography~$E$ may be taken to be ``at infinity'' in the sense that broken rays are allowed to tend to infinity.
Integrability assumptions are then needed for the unknown functions.

\subsection{Remarks}
\label{sec:rmk}

It is important to notice that the gluing in Proposition~\ref{prop:ex-cone} was done along flat parts of the boundary (line segments in the plane).
A particular case of Proposition~\ref{prop:ex-cone}(3) is the half space $\R^{n-1}\times[0,\infty)$ (with $k=\infty$); the reflection is simply obtained by $\rfl{f}(x',x)=f(x',\abs{x})$ and there are no corners (or the corners have angle~$\pi$).
We will reflect and glue together manifolds in Section~\ref{sec:rfl} below, and flatness of the gluing boundary will lead to regularity of the reflected manifold (see Lemma~\ref{lma:rfl-reg}).
We recall that in Isakov's reflection method for the Calder\'on problem~\cite{I:refl-calderon} reflection is made along a part of the boundary which is flat (hyperplane) or can be conformally flattened (sphere).

The recent result by Hubenthal in the square~\cite{H:square} heavily relies on the geometry of the square: reflections are done at straight lines and corners have angle~$\pi/4$.
By Proposition~\ref{prop:ex-cone}(1) the broken ray transform is injective in the square, provided that the set of tomography~$E$ contains two adjacent edges of the square.
Similarly the broken ray transform is injective in any polygon if the reflecting part~$R$ of the boundary contains at most two adjacent edges.
Although these results are different in their formulation and methods of proof, the underlying geometrical structure of the square is heavily relied on.

The shape of the domain $\Omega$ in dimensions three and higher can be other than the cone in Proposition~\ref{prop:ex-cone}(5).
For example, if $\Omega$ is a cube with three adjacent faces as the set of tomography, eight copies of it can be glued together to form a bigger cube in a fashion similar to gluing four squares to form a bigger square in dimension two.
(Such a construction is used to prove Proposition~\ref{prop:pbrt-cube}.)
The correspondence between lines and broken rays is the same.
We do not elaborate on all the possibilities here; we only wish to present the idea in a fair amount of generality.

In the discussion below we will focus on the analogue of the half space.
Corners can be allowed, but for the sake of simplicity we shall not allow them.
It is the author's belief that if the corners add up nicely as in Proposition~\ref{prop:ex-cone}(1), Theorem~\ref{thm:rfl} remains true.
The technical difficulty lies in the fact that one needs some kind of ``corner normal coordinates'' at a corner point of a manifold.
For more general corners one needs something to replace Helgason's support theorem in the proof of Proposition~\ref{prop:ex-cone}(2).
Support theorems as simple and powerful as Helgason's seem not to be available for general manifolds.

We will, however, use support theorems on manifolds for a different purpose.
The results of~\cite{UV:local-x-ray} and~\cite{K:spt-thm} are therefore given in Section~\ref{sec:spt-corner}.

The examples above were such that the domain $\Omega$ was constructed as a submanifold of a particularly nice domain in $\R^2$ so that the construction reduces the problem to a planar one.
The examples in Section~\ref{sec:ex} are also of this type.
If we start with an arbitrary manifold, the resulting manifold is not generally any simpler than the one we started with.

If the angle in Proposition~\ref{prop:ex-cone}(2) is a reflex angle (between~$\pi$ and~$2\pi$), there is no need for a reflection construction.
Helgason's support theorem immediately gives injectivity for the broken ray transform, and reflected rays need not be considered at all.
This observation holds true whenever the reflector is concave.

The cone in Proposition~\ref{prop:ex-cone}(2) need not have an angle.
If the domain looks like the domain of Proposition~\ref{prop:ex-cone}(2) outside some neighborhood of the origin, we can reconstruct~$f$ outside some (possibly larger) neighborhood of the origin.
This can be done with the same method; Helgason's support theorem tells that~$\rfl{f}$ vanishes outside the convex hull of the set (and its copies) where the boundary of~$\Omega$ is not conical.

\section{Two applications on manifolds}
\label{sec:appl}

We use Theorem~\ref{thm:rfl} to give two theorems of injectivity of the broken ray transform on in a fairly large class of manifolds.
In brief, Theorem~\ref{thm:rfl} tells that injectivity of the broken ray transform on a manifold can be reduced to the injectivity of the geodesic ray transform on a reflected manifold in analogue to the half plane example given in the beginning of this article.
The notation and necessary results are given in Section~\ref{sec:rfl} below.
The purpose of this section is to motivate the general construction.
More examples are given in Section~\ref{sec:ex}.

In the proofs below, $\rfl{A}$ is a doubled version of the manifold~$A$ obtained by reflecting with respect to~$\bar{R}$.
The construction is illustrated in Fig.~\ref{fig:rfl-constr} and given in detail in Section~\ref{sec:rfl-constr}.

%As an application of Theorem~\ref{thm:rfl} we prove the following two results.

\begin{theorem}
\label{thm:brt2}
Let $M$ satisfy the following assumptions:
\begin{itemize}
\item $M$ is a smooth Riemannian surace with corners, and the boundary is a disjoint union of the sets~$E$, $R$, and~$C$.
\item The open smooth boundary components~$E$ and~$R$ meet orthogonally at~$C$.
\item $E$ is strictly convex and~$R$ is $\infty$-flat.
\item The local boundary defining functions of~$M$ near~$C$ can be chosen to be $\infty$-even at~$R$.
%\item The boundary $\partial A$ is strictly convex and meets $\partial M$ orthogonally and $A\cap\partial M$ is $\infty$-flat in the sense of Definition~\ref{def:even-flat}.
\item For any two points on~$\bar{E}$ (but not both of them on $C$) and a chosen parity (odd or even), there is a unique broken ray in~$M$ with the chosen parity (even or odd number of reflections) joining the points, and this geodesic depends smoothly on the endpoints.
\item If both endpoints lie in~$C$, the geodesic is contained in~$R$.
\item The normal derivatives of odd orders with respect to endpoints of the geodesic vanish at~$C$.
\end{itemize}

%Let $M$ be a Riemannian surface and let $A\subset M$ be a closed subset with smooth boundary such that the following conditions are satisfied.
%The boundary $\partial A$ is strictly convex and meets $\partial M$ orthogonally and $A\cap\partial M$ is $\infty$-flat in the sense of Definition~\ref{def:even-flat}.
%In addition, for any two points on $\partial A$ and a chosen parity (odd or even), there is a unique broken ray in $A$ with the choicen parity (even or odd number of reflections) joining the points, and this geodesic depends smoothly on the endpoints.

%Then the broken ray transform in $\bar{A}$ is injective with set of tomography $E=\partial A$ (in the topology of $M$) for smooth functions in $A$ which are $\infty$-even at $A\cap\partial M$.
Then the broken ray transform in~$\bar{M}$ is injective with set of tomography~$E$ for smooth functions in~$M$ which are $\infty$-even at~$R$.
\end{theorem}

\begin{proof}
By Theorem~\ref{thm:rfl} it is enough to show that the geodesic ray transform $\rt{}$ is injective on the Riemannian manifold $(\rfl{A},\rfl{g})$ in the class~$C^\infty(\rfl{A})$.

The reflected manifold $\rfl{A}$ is simple by Lemma~\ref{lma:rfl-reg}(10), so the geodesic ray transform is injective by~\cite[Theorem~1.1]{SU:surface}.
\end{proof}

\begin{theorem}
\label{thm:brt3}
Let $(M,g)$ be a compact Riemannian manifold with boundary and suppose that $\dim M\geq3$.
Assume $\rho\in C^\infty(M,\R)$ satisfies the following:
\begin{enumerate}
\item $d\rho\neq0$ on $\rho^{-1}(0)$.
\item $\rho$ is $\infty$-even at $\partial M\cap\rho^{-1}([0,\infty))$.
\item If $T=\max_M\rho$, the set $\rho^{-1}(T)$ has zero measure.
\item The level set $\rho^{-1}(t)$ is strictly convex in $\rho^{-1}((t,\infty))$ for all $t\in(0,T)$.
\item $\partial M\cap\rho^{-1}([0,\infty))$ is $\infty$-flat.
\end{enumerate}
Denote $A=\rho^{-1}([0,\infty))$.

Then the broken ray transform is injective in the class $L^2(A)$ in~$A$, when the set of tomography is $E=\rho^{-1}(0)\subset\partial A$.
\end{theorem}

To prove the theorem, we need the following result.

\begin{theorem}[Corollary in \cite{UV:local-x-ray}]
\label{thm:UV}
Let $(X,g)$ be a compact Riemannian manifold of dimension at least 3 with boundary embedded in a manifold $(\ext{X},g)$.
Assume there is a function $\rho\in C^\infty(\ext{X},\R)$ such that $X=\rho^{-1}([0,\infty))$, $\partial X=\rho^{-1}(0)$, $d\rho\neq0$ on~$\partial X$.
Let $T=\max_{\ext{X}}\rho$. Assume furthermore that $\rho^{-1}([t,\infty))$ is strictly convex for all $t\in[0,T)$ and $\rho^{-1}(T)$ has zero measure. %(resp. empty interior)

Then the geodesic ray transform is injective in the class $L^2(X)$ on the manifold~$X$. %(resp. $H^s(X)$ for $s>n/2$)
\end{theorem}

\begin{proof}[Proof of Theorem~\ref{thm:brt3}]
By Theorem~\ref{thm:rfl} it is enough to show that geodesic ray transform is injective on $\rfl{A}$ for the class $\rfl{L^2(A)}$.
By Lemma~\ref{lma:rfl-reg} the manifold $\rfl{A}$ satisfies the assumptions of Theorem~\ref{thm:UV} and also $L^2(\rfl{A})\subset\rfl{L^2(A)}$, which confirms the claim.
\end{proof}

\begin{remark}
\label{rmk:stab3}
A stability estimate for Theorem~\ref{thm:UV} is given in~\cite{UV:local-x-ray}.
That estimate immediately yields a stability estimate for Theorem~\ref{thm:brt3}.
\end{remark}

\section{Reflected manifolds}
\label{sec:rfl}

The key idea in the proof of Proposition~\ref{prop:ex-cone} was to glue together copies of the original conical domain~$\Omega$.
The most simple case was when~$\Omega$ was part of a half space and the reflecting part of $\partial\Omega$ lay on the boundary of the half space.
In this case two copies of the original domain~$\Omega$ could be glued together to form~$\rfl{\Omega}$.

Similar reflecting and gluing can be done for Riemannian manifolds.
We focus here on the case analogous with the Euclidean half space.
The construction for more complicated Euclidean domains presented above can be generalized in the same fashion, but for the sake of simplicity we omit them here.
The analogue of Proposition~\ref{prop:ex-cone} for Riemannian manifolds can be used to show injectivity of the broken ray transform on some Riemannian manifolds.

\subsection{Construction of reflected manifolds}
\label{sec:rfl-constr}

\begin{figure}%
\includegraphics[width=\columnwidth]{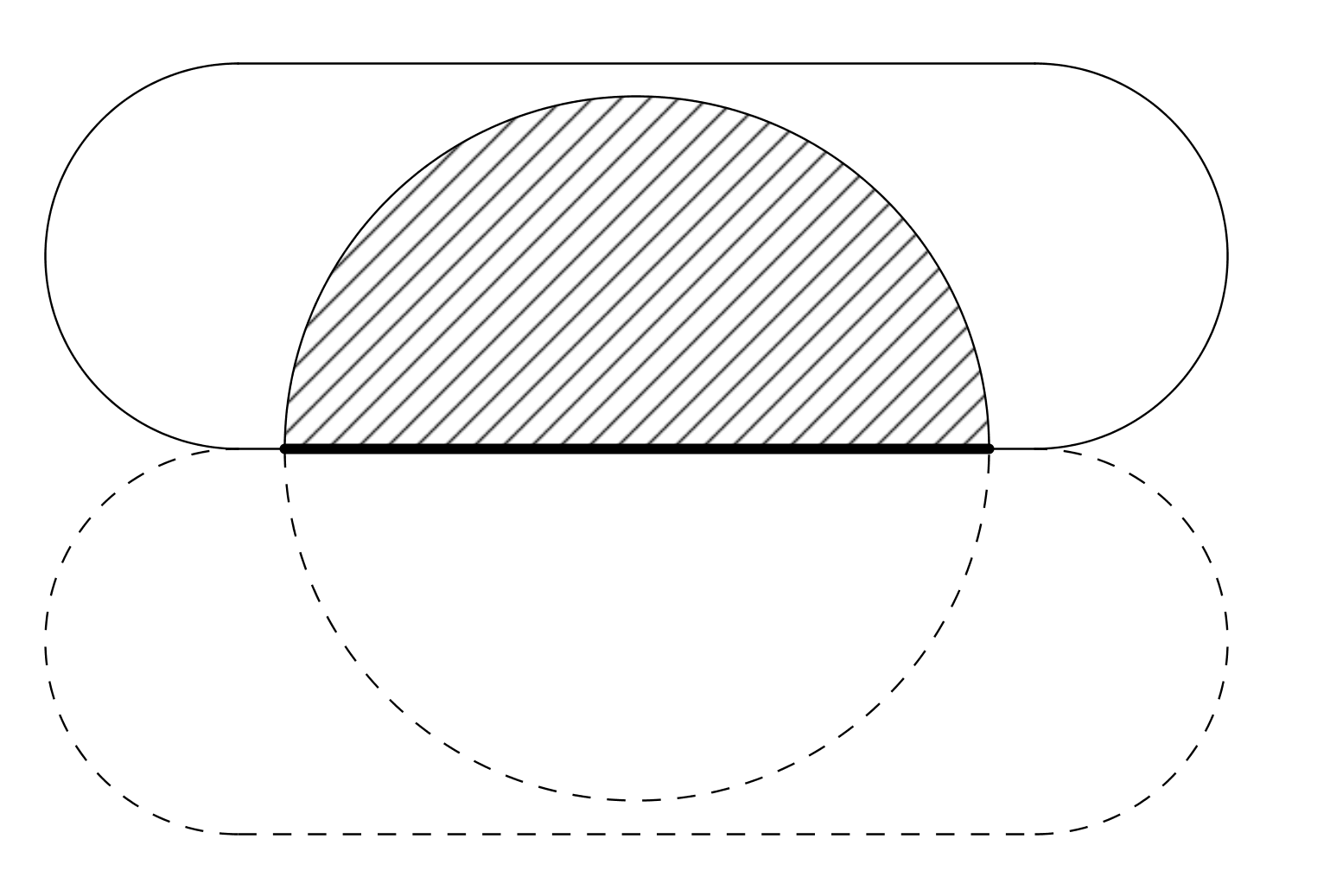}
%\psset{xunit=1.0cm,yunit=1.0cm,algebraic=true,dotstyle=o,dotsize=3pt 0,linewidth=0.8pt,arrowsize=3pt 2,arrowinset=0.25}
%\begin{pspicture*}(-3.01,-3.07)(10.03,5.72)
%\parametricplot{1.5707963267948966}{4.71238898038469}{1*1.89*cos(t)+0*1.89*sin(t)+-0.67|0*1.89*cos(t)+1*1.89*sin(t)+3.21}
%\psline(-0.67,5.09)(7.13,5.09)
%\parametricplot[linestyle=dashed,dash=5pt 5pt]{1.5707963267948966}{4.71238898038469}{-1*1.89*cos(t)+0*1.89*sin(t)+7.13|0*1.89*cos(t)+-1*1.89*sin(t)+-0.57}
%\parametricplot[linestyle=dashed,dash=5pt 5pt]{1.5707963267948966}{4.71238898038469}{1*1.89*cos(t)+0*1.89*sin(t)+-0.67|0*1.89*cos(t)+-1*1.89*sin(t)+-0.57}
%\parametricplot{1.5707963267948966}{4.71238898038469}{-1*1.89*cos(t)+0*1.89*sin(t)+7.13|0*1.89*cos(t)+1*1.89*sin(t)+3.21}
%\psline[linestyle=dashed,dash=5pt 5pt](-0.67,-2.45)(7.13,-2.45)
%\parametricplot[hatchcolor=black,fillstyle=hlines,hatchangle=45.0,hatchsep=0.2]{0.0}{3.141592653589793}{1*3.45*cos(t)+0*3.45*sin(t)+3.23|0*3.45*cos(t)+1*3.45*sin(t)+1.32}
%\parametricplot[linestyle=dashed,dash=5pt 5pt]{-3.141592653589793}{0.0}{1*3.45*cos(t)+0*3.45*sin(t)+3.23|0*3.45*cos(t)+1*3.45*sin(t)+1.32}
%\psline(-0.67,1.32)(7.13,1.32)
%\psline[linewidth=4pt](-0.22,1.32)(6.68,1.32)
%\end{pspicture*}
\caption{An example of the construction in Section~\ref{sec:rfl-constr}. Here the underlying manifold~$M$ is the solid stadium shaped domain and $A$ is the shaded half disk within. The set~$R$ is the thick line, $E$ is the upper half circle and~$C$ constists of the two corner points in~$A$. The manifold $A$ is reflected into a doubled manifold~$\rfl{A}$, which here is the full disk. The resulting manifold~$\rfl{A}$ is a manifold with smooth boundary.}%
\label{fig:rfl-constr}%
\end{figure}

In the following $(M,g)$ is a smooth compact Riemannian manifold with boundary and $A=\overline{\sisus A}\subset M$ is a closed subset with $A\cap\partial M\neq\emptyset$ and the $C^1$ boundary of $A$ meets $\partial M$ orthogonally.
We reflect the set~$A$ with respect to $\partial M\cap A$.
The manifold~$\rfl{A}$ constructed below is this reflection of~$A$.
The construction is illustrated if Figure~\ref{fig:rfl-constr}.

Let $A_1=A$, $C_1=\partial M\cap\partial A$, $R_1=\partial M\cap\sisus A$, and $E_1=\sisus M\cap\partial A$.
The ``boundary'' (it is not the boundary in the topology of $M$) of $A_1$ is the (disjoint) union $\partial A_1=R_1\cup C_1\cup E_1$.
Now $A_1$ is a topological manifold with boundary with and the boundary is~$\partial A_1$.
Furthermore, $A_1$ is a smooth manifold with corners with
interior $\sisus A_1=A_1\setminus\partial A_1$,
smooth part of the boundary $R_1\cup E_1$,
and nonsmooth part of the boundary $C_1$.
The sets $\sisus A_1=A_1\setminus\partial A_1$, $R_1\cup E_1$ and~$C_1$ are the strata of~$A_1$ of depths 0, 1 and 2, respectively, in the sense of~\cite{J:man-corner}.
%We note that $\bar{R}_1=R_1\cup C_1$.

By $\partial A$ we always mean the boundary of $A$ in the topology of~$M$.
Thus, if we identify~$A$ with~$A_1$, we have $\partial A\subsetneq\partial A_1$.

The higher depth strata of~$A_1$ are empty (there can only be a corner in one direction), but in the example of Proposition~\ref{prop:pbrt-cube} strata of all possible depths appear.
We bound the depth of strata by 2 for technical simplicity.
As manifolds with smooth boundary are more convenient to work with, we will reduce the depth bound to 1 in Section~\ref{sec:spt-corner}.

We define $A_2$ as an identical copy of $A_1$ with all labels changed, and glue $A_1$ and~$A_2$ together along $R_1$ and~$R_2$ to form a manifold~$\rfl{A}$ with (smooth) boundary.

Let $\eta:A_1\to A_2$ be the natural bijection.
We define the relation $\sim$ on the disjoint union $A_1\cup A_2$ by letting $x\sim y$ if $x=y$, or $x\in \bar{R}_1$ and $y=\eta(x)$, or $y\in \bar{R}_1$ and $x=\eta(y)$.
This is obviously an equivalence relation, and the quotient space $\rfl{A}=(A_1\cup A_2)/\sim$ is well defined.
We denote by $\rfl{R}$ the image of $R_1$ (or $R_2$; the image is the same) under the quotient map.
We consider $A_i\setminus \bar{R}_i$, $i=1,2$, to be a subspace of~$\rfl{A}$ in the natural way.
We define $\iota_i:A_i\to\rfl{A}$ to be the natural injection.

It is geometrically rather obvious that~$\rfl{A}$ is an $n$-dimensional topological manifold with boundary.
The boundary of $\rfl{A}$ is a disjoint union of $\iota_1(E_1)$, $\iota_2(E_2)$, and $\rfl{C}=\iota_1(C_1)=\iota_2(C_2)$, and the interior is a disjoint union of $\iota_1(\sisus A_1)$, $\iota_2(\sisus A_2)$, and $\rfl{R}=\iota_1(R_1)=\iota_2(R_2)$.
Using Riemannian boundary normal coordinates at $R_i$ we also turn it into a smooth manifold with boundary in a natural way.
In $C_i$ we may proceed similarly, but the model space is $\R^{n-2}\times[0,\infty)^2$ which after reflection and gluing becomes $\R^{n-1}\times[0,\infty)$; we have $\rfl{C}\subset\partial\rfl{A}$ as expected.

Due to the use of Riemannian boundary normal coordinates the transition maps are smooth at~$\rfl{R}$, and~$\rfl{A}$ is indeed a smooth manifold with boundary.

The natural projection map $\pi_1:\rfl{A}\to A_1$ defined by $\pi(\iota_1(x))=x$ and $\pi(\iota_2(x))=\eta^{-1}(x)$ for each $x\in A_1$ is a covering map.
Identifying~$A$ with~$A_1$ and~$A_2$ and writing $\pi:\rfl{A}\to A$ as the projection, we have the obvious property
\begin{equation}
\label{eq:proj-inj}
\pi\circ\iota_i=\id:A\to A\text{ when }i=1,2.
\end{equation}
The projection $\pi$ can be used to pull back (scalar, vector, and tensor) functions from~$A$ to the reflected manifold~$\rfl{A}$; we define for any $f:A\to\C$ the reflected version $\rfl{f}=\pi^*f$ and similarly for higher rank tensors.

Special care must be taken since~$\pi$ is not smooth but only continuous on~$\rfl{R}$. This is related to the fact that not all choices of boundary charts at $R_1$ give $\rfl{A}$ a smooth structure. Some additional conditions thus need to be satisfied to guarantee that~$\rfl{f}$ is smooth if~$f$ is. This issue is considered in Section~\ref{sec:rfl-reg} in more detail.

%Identifying $A$ with $A_1$ and $A_2$, we have the obvious property
%\begin{equation}
%\label{eq:proj-inj}
%\pi\circ\iota_i=\id:A\to A\text{ when }i=1,2.
%\end{equation}

The above construction does not use the smoothness of the manifold~$M$ and its metric~$g$.
If we instead equip~$M$ with $C^k$ differentiable structure and take $g\in C^{k,\alpha}$ with $k+\alpha\geq2$, then $\rfl{A}$ is naturally equipped with a~$C^k$ structure.
The condition $k+\alpha\geq2$ ensures that geodesics on~$M$ do not branch, and the boundary normal coordinates actually provide coordinates.

We remark that the set~$A$ is not a manifold with boundary, since it is not locally diffeomorphic to $\R^{n-1}\times[0,\infty)$ at points in $\partial A\cap\partial M$.
At these points the proper model space is $\R^{n-2}\times[0,\infty)^2$, which makes $A$ a manifold with corners in the sense defined in~\cite{J:man-corner}.
Theorem~\ref{thm:rfl} for manifolds with corners will be generalized to manifolds without corners in Section~\ref{sec:spt-corner}.

We wish to point out that both precomposition (for functions on a manifold) and inverse of postcomposition (for curves) of an object with the projection~$\pi$ are denoted by a tilde.
An object with tilde should therefore be understood as the natural corresponding object (or one of them in the case of curves) on the reflected manifold~$\rfl{A}$.

\subsection{Regularity of the reflected manifold}
\label{sec:rfl-reg}

We keep the assumptions made in the beginning of Section~\ref{sec:rfl-constr} regarding~$M$, $g$ and~$A$ but the smoothness requirement is only that $g\in C^{k,\alpha}$ with $k+\alpha\geq2$. Lemma~\ref{lma:rfl-reg} below demonstrates the correspondence between the properties of~$A$ and its reflected version~$\rfl{A}$.

\begin{definition}
\label{def:even-flat}
Let $B\subset\partial M$ and $k\leq m$.
A function $f\in C^m(M)$ is $k$-even at~$B$ if $\partial_\nu^if|_B=0$ for all odd $i\leq k$.

A rank two tensor $f_{ij}$ (written in boundary normal coordinates near the boundary) of class $C^m$ is $k$-even at~$B$ if the functions $f_{11}$ and $f_{ij}$ for $i,j>1$ are $k$-even at~$B$ in the above sense and $\partial_\nu^if_{1i}|_B=\partial_\nu^if_{i1}|_B=0$ for all even~$i\leq k$.
In particular, $B\subset\partial M$ is called $k$-flat if the metric~$g$ is $k$-even at~$B$.
\end{definition}

The above definition can be easily extended to tensors of any rank, but we do not need to consider ranks other than zero and two here.
The definitions are given so that a~$C^k$ tensor field~$f$ on $A$ is $k$-even at~$B$ if and only if $\rfl{f}=\pi^*f$ is a~$C^k$ tensor field on~$\rfl{A}$.
This correspondence is the basis of Lemma~\ref{lma:rfl-reg} below.

\begin{definition}
Let $A\subset M$ be an closed subset of~$M$ as in Section~\ref{sec:rfl-constr}.
A set~$B\subset\partial A$ (boundary in the topology of~$M$) is evenly (resp. oddly) strictly broken ray convex in~$A$, if for any two points on~$B$ there is a broken ray with an even (resp. odd) number of reflections (possibly zero) connecting the two points such that the interior of the broken ray is in~$A$.
\end{definition}

\begin{lemma}
\label{lma:rfl-reg}
Regularity of functions and metrics on the original manifold~$M$ and the reflecting manifold~$\rfl{M}$ correspond in the following way (here $k,m\in\{0,1,\dots,\infty\}$):
\begin{enumerate}
\item $R$ is 1-flat if and only if the second fundamental form vanishes on~$R$.
\item If $g\in C^k$ and $R$ is $m$-flat, then $\rfl{g}\in C^{\min(k,m)}$.
\item If $f\in C^k$ and $f$ is $m$-even at $R$, then $\rfl{f}\in C^{\min(k,m)}$.
\item If $f\in L^p$, then $\rfl{f}\in L^p$.
\item If $A$ is (strictly) evenly and oddly broken ray convex, then~$\rfl{A}$ is geodesically convex.
\item If $A$ is strictly convex , then~$\rfl{A}$ is strictly convex.
\item If $A$ is convex, then~$\rfl{A}$ is convex.
\item Suppose $g$ and~$R$ are~$C^3$. If~$R$ is strictly convex, there are no geodesics tangent to~$\rfl{R}$.
\item If $R$ is strictly concave, geodesics tangent to~$\rfl{R}$ branch.
\item $\rfl{M}$ is simple if the assumptions listed in Theorem~\ref{thm:brt2} hold (with the word `surface' replaced by `manifold').
\end{enumerate}
In the parts~(8--9) a geodesic means a locally length minimizing curve, since the geodesic equation does not make sense on~$\rfl{R}$ when~$R$ is not 1-flat.
\end{lemma}
\begin{proof}
(1) Use boundary normal coordinates with the first coordinate as the normal direction to the boundary.
In these coordinates $g_{11}=1$ is constant and $g_{1i}=0$ when~$i>0$.
Thus it suffices to study $\partial_\nu g_{ij}$ for all~$i,j>1$.

In these coordinates $\nu=(-1,0,\dots,0)$ is the outward unit normal vector.
Since~$\nu$ is constant in these coordinates, for two vectors~$a$ and~$b$ tangent to the boundary at a boundary point we have for the second fundamental form
\begin{equation}
%\label{eq:}
\begin{split}
\sff{a}{b}
&=
-\ip{\nabla_a\nu}{b}
\\&=
-(a^j\partial_j\nu^i+a^j\nu^k\Gamma^i_{\phantom{i}jk})b_i
\\&=
\Gamma^i_{\phantom{i}j1}a^jb_i
\\&=
\frac{1}{2}(\partial_1g_{ij}+\partial_jg_{i1}-\partial_ig_{j1})a^jb^i
\\&=
\frac{1}{2}\partial_1g_{ij}a^jb^i.
\end{split}
\end{equation}
Thus $\sff{\cdot}{\cdot}$ vanishes on $R$ if and only if $\partial_\nu g_{ij}=0$ for all~$i,j>1$.

(2--7) Obvious.

%(3) Proven essentially like~(2).
%
%(4--7) Obvious.
%For claims~(6--7), note that $\partial A$ meets $\partial M$ orthogonally.

(8) Let $\rho:\rfl{A}\to\rfl{A}$ be the map which reflects~$\rfl{A}$ with respect to~$\rfl{R}$ in the natural way.

Suppose $\gamma:(-\delta,\delta)\to\rfl{A}$ is a geodesic which meets~$\rfl{R}$ tangentially at~$\gamma(0)$.
The intersection points of~$\gamma$ and~$\rfl{R}$ cannot accumulate at~$\gamma(0)$ unless $\gamma((-\delta,\delta))\subset\rfl{R}$;
the points of reflection of a broken ray of finite length near a strictly convex part of the boundary cannot accumulate.
The proof of this statement is too long to be included here; see~\cite{I:bdy-det} for proof and explanation of the~$C^3$ assumption.
%Choosing $\delta$ small enough, either $\gamma|_{[0,\delta)}$ stays on one side of $\rfl{R}$ (say, $\iota(A_1)$), or for any $\eps\in(0,\delta)$ there are times $t_1,t_2\in(0,\eps)$ with $\gamma(t_i)\in\iota(A_i)$.
%
%In the first case $\gamma$ cannot be a geodesic by strict convexity of $R_1$.
%
%In the second case, we can find $t_1\in(0,\eps)$ for any $\eps\in(0,\delta)$ with $\gamma(t_1)\in\iota(A_1)$ and $t_2\in(0,t_1)$ with $\gamma(t_2)\in\iota(A_2)$ and $t_1'\in(0,t_2)$ with $\gamma(t_1')\in\iota(A_1)$.
%Taking $\eps$ small enough, the unique geodesic joining $\gamma(t_1)$ and $\gamma(t_1')$ has to lie in $\iota(A_1)$ due to strict convexity, which gives a contradiction.

There are three options left:
(a)~$\gamma$ intersects $\rfl{R}$ only at~$\gamma(0)$ and stays on one side of~$\rfl{R}$ (say,~$\iota(A_1)$),
(b)~$\gamma$ intersects $\rfl{R}$ only at~$\gamma(0)$ and changes side there (we may choose $\gamma(t)\in\iota(A_1)$ for~$t\leq0$ and $\gamma(t)\in\iota(A_2)$ for $t\geq0$), or
(c)~$\gamma$ lies in~$\rfl{R}$.

In case~(a) $\gamma$ cannot be a geodesic because of strict convexity of~$R_1$.
In case~(b) define a curve $\phi$ as $\phi(t)=\gamma(t)$ for $t\leq0$ and $\phi(t)=\rho(\gamma(t))$ for~$t>0$.
By construction of the reflected manifold, the curve $\phi$ is also a geodesic.
But now~$\phi$ falls in the case~(a), which is impossible.
Also case~(c) is impossible, since a curve lying at a strictly convex subset~$R$ of the boundary $\partial M$ cannot be a geodesic.

We conclude that a geodesic tangential to~$\rfl{R}$ at a point where~$R_1$ is strictly convex cannot exist.

(9) Suppose $\gamma:(-\delta,\delta)\to\rfl{A}$ is a geodesic which meets~$\rfl{R}$ tangentially at~$\gamma(0)$.
Consider the case when $\gamma(t)\notin\rfl{R}$ for~$t>0$.
Now construct another geodesic $\phi:(-\delta,\delta)\to\rfl{A}$ by letting $\phi(t)=\gamma(t)$ for $t\leq0$ and $\phi(t)=\rho(\gamma(t))$ for~$t>0$.
By the construction of the reflected manifold~$\rfl{A}$ also~$\phi$ is a geodesic.
Thus $\gamma$ branches at~$t=0$.

Then consider the case when $\gamma(t)\in\rfl{R}$ for $t\geq0$ or the points where $\gamma|_{[0,\delta)}$ intersects~$\rfl{R}$ accumulate at $\gamma(0)$.
Now define $\phi$ as $\gamma$ for $t\leq0$ as above and let~$\phi$ for~$t>0$ be the unique geodesic in $\sisus\iota(A_1)$ with initial direction $\dot{\gamma}(0)$.
This geodesic exists if~$\delta$ is small enough.
Again, the curve $\phi$ is a geodesic and~$\gamma$ branches.

We conclude that any geodesic tangent to~$\rfl{R}$ at a strictly concave point always has nonunique continuation.

(10)
The assumptions imply that~$\rfl{M}$ is smooth and has smooth and strictly convex boundary.
Also for any two boundary points there is a unique geodesic joining them and the geodesic depends smoothly on its endpoints.
Thus~$\rfl{M}$ is simple by definition.
\end{proof}

\begin{remark}
To have unique geodesics on~$\rfl{M}$, we want~$\rfl{g}$ to be $C^{1,1}$~(or $C^2$).
By the above lemma, for this we need that the original metric~$g$ is $C^{1,1}$ or $C^2$ and the reflector~$R$ is flat in the sense that the second fundamental form vanishes.
For higher regularity of~$\rfl{g}$ we need higher order flatness of the reflector.
\end{remark}

For shorthand, we give the following definition so that the various cases of Lemma~\ref{lma:rfl-reg} need not be listed again when it is used.

\begin{definition}
If~$F$ is class of functions from~$A$ to~$\R$ (e.g.~$F=C^k(A,\R)$ or $F=L^p(A,\R)$), we define the reflected class of functions by
\begin{equation}
%\label{eq:}
\rfl{F}=\{\rfl{f}:f\in F\}.
\end{equation}
\end{definition}

\subsection{From broken ray transform to geodesic ray transform}

The main result we present is Theorem~\ref{thm:rfl}.
It is a direct generalization of the ideas behind the proofs in Proposition~\ref{prop:ex-cone}, but we state it as a theorem to highlight the generality of the reflection construction.
We gave two applications of this theorem in Section~\ref{sec:appl} to show injectivity of the broken ray transform on a fairly large class of manifolds.
Simpler and more concrete examples are given in Section~\ref{sec:ex}.
The geodesic ray transform and the broken ray transform were defined in definition~\ref{def:rt-brt}.

%xxx Old position of definition \label{def:rt-brt}

We remind the reader that the set~$A$ is a manifold with corners.
The case of manifolds with smooth boundary requires more work and will be discussed in Section~\ref{sec:spt-corner}.

\begin{theorem}
\label{thm:rfl}
Let $A$ be as in Section~\ref{sec:rfl-constr}.
Let $F,H:A\to\R$ be some classes of functions on~$A$ and let $E=\partial A\setminus\partial M$ be the set of tomography. Then:
\begin{enumerate}
\item If $\rt{\rfl{h}}\rfl{f}$ determines both $\rfl{h}\in\rfl{H}$ and $\rfl{f}\in\rfl{F}$, then $\brt{h}f$ determines both $h\in H$ and $f\in F$.
\item If $\rt{\rfl{h}}\rfl{f}$ determines $\rfl{f}\in\rfl{F}$ for a fixed (known) $\rfl{h}\in\rfl{H}$, then $\brt{h}f$ $f\in F$ for a fixed (known)~$h\in H$.
\end{enumerate}
\end{theorem}
\begin{proof}
We only prove part~(1); part~(2) results by letting $H=\{h\}$.
Suppose $\rt{\rfl{h}}\rfl{f}$ indeed determines both $\rfl{h}\in\rfl{H}$ and $\rfl{f}\in\rfl{F}$, and that $\brt{h}f$ is given.

Let $f\in F$ and $h\in H$. Construct~$\rfl{A}$ as in Section~\ref{sec:rfl-constr}.
Take any geodesic $\phi\in\Gamma(\rfl{A})$.
If we let $\gamma=\pi\circ\phi$, we have $\gamma\in\Gamma_E(A)$.
By definition of $\gamma$ and the reflected functions we have $\rfl{f}\circ\phi=f\circ\gamma$ and $\rfl{h}\circ\phi=h\circ\gamma$, so (cf.~\eqref{eq:rfl-int})
\begin{equation}
%\label{eq:}
\rt{\rfl{h}}\rfl{f}(\phi)=\brt{h}f(\gamma).
\end{equation}
Thus, from the given function $\brt{h}f$ we obtain $\rt{\rfl{h}}\rfl{f}$. This determines~$\rfl{h}$ and~$\rfl{f}$ by assumption, and by Eq.~\eqref{eq:proj-inj} this determines~$h$ and~$f$.
\end{proof}

\begin{remark}
Any stability result for the geodesic ray transform on~$\rfl{A}$ immediately yields a stability result for the broken ray transform on~$A$.
Since stability is inherited in such a way, we do not discuss the stability of the broken ray transform on different manifolds.
\end{remark}

\begin{remark}
The above theorem only considers scalar functions $f:M\to\R$.
We can similarly define the broken ray transform for a tensor field of any order just like one defines the geodesic ray transform for a tensor field.
The theorem holds true for tensor fields as well, and the proof is the same; a tensor field~$f$ is reflected to $\rfl{f}=\pi^*f$ instead of simply reflecting all the component functions.
(A tensor function can only be recovered up to the natural gauge freedom; see e.g.~\cite{PSU:tensor-surface}.)
The theorem also remains true if one introduces a weight in the broken ray transform.
Replacing real numbers with complex numbers is also a trivial generalization.
\end{remark}

\begin{remark}
\label{rmk:ex-att}
The examples in Proposition~\ref{prop:ex-cone} were concerned with zero attenuation.
The attenuated broken ray transform is injective provided the corresponding attenuated ray transform in the plane is injective.
The analogue of Helgason's support theorem holds true with constant attenuation~\cite[Theorem~4.2]{K:exp-radon}.
For more results on attenuated ray transforms in Euclidean spaces, we refer to~\cite{N:attenuated-x-ray,BS:attenuated-x-ray,SF:attenuated-x-ray,F:attenuated-x-ray}.
Attenuated transforms have also been considered on manifolds (see e.g.~\cite{S:attenuated-x-ray,SU:attenuated-x-ray}), but we set our focus on the nonattenuated setting.
\end{remark}

\section{Support theorems and manifolds without corners}
\label{sec:spt-corner}

Theorem~\ref{thm:rfl} above was stated for a manifold~$A$ with corners.
It is appealing to consider the broken ray transform on a manifold~$M$ with smooth boundary (that is, without corners).
To achieve this we take the following two steps:
First, using geodesics (broken rays without reflections) with endpoints in~$E$ one can in favorable situations recover the unknown function~$f$ in a neighborhood~$V$ of~$E$.
(We refer to results of this nature as ``support theorems'' since they in a way generalize Helgason's support theorem.)
Second, if this neighborhood is nice enough, $A=M\setminus V$ is a manifold with corners and Theorem~\ref{thm:rfl} is applicable.

Such generalized support theorems for manifolds are given below in Theorems~\ref{thm:UV-local} and~\ref{thm:K-spt}.
Using these, a suitable form of this support principle for broken ray transform is given in Theorem~\ref{thm:brt-spt}.

We do not formulate this two step procedure as a theorem since the geometry of the set~$V$ is difficult to control in terms of assumptions on~$M$ and~$E$.
A specific example where this idea works is given in Proposition~\ref{prop:ex-thm}(3).

%\subsection{Support theorems}
%\label{sec:spt-thm}

\begin{theorem}[Theorem in \cite{UV:local-x-ray}]
\label{thm:UV-local}
Let $M$ be a manifold with boundary with dimension 3 or greater.
If $\partial M$ is strictly convex at $p\in\partial M$, then there is a neighborhood~$O$ of~$p$ such that the geodesic ray transform is injective in~$O$ in the following sense:
If the integral of a function in $f\in L^2(M)$ vanishes over all geodesics with interior in~$O$ and endpoints in $O\cap\partial M$, then~$f$ vanishes in~$O$.
\end{theorem}

\begin{theorem}[\cite{K:spt-thm}]
\label{thm:K-spt}
Let $(M, \partial M, g)$ be a simple Riemannian manifold embedded in a slightly larger manifold $(\ext{M},\partial\ext{M},\ext{g})$ and assume that the metric~$\ext{g}$ is real analytic.
Let $\mathcal{A}$ be an open set of geodesics in~$\ext{M}$ such that that each geodesic $\gamma\in\mathcal{A}$ can be deformed to a point on the boundary $\partial\ext{M}$ by geodesics in~$\mathcal{A}$.
Let $M_\mathcal{A}$ be the set of points lying on the intersection of these geodesics with~$M$.
If $f\in L^2(M)$ is a function such that the integral of~$f$ is zero over every geodesic in $M$ that has an extension in $\mathcal{A}$, then $f=0$ on~$M_\mathcal{A}$.
\end{theorem}

\begin{theorem}
\label{thm:brt-spt}
Let $M$ be a manifold with boundary and $E\subset\partial M$ the set of tomography.
Suppose any one of the following:
\begin{enumerate}
\item $n\geq3$ and $E$ is open and strictly convex.
\item The metric is analytic and the manifold simple, the set~$E$ is open and strictly convex, and $M$ can be extended to slightly larger manifold~$\ext{M}$.
Any geodesic in~$M$ with endpoints in~$E$ can be extended to a geodesic $\ext{\gamma}$ in $\ext{M}$ and~$\ext{\gamma}$ can be deformed to a point on $\partial\ext{M}$ by geodesics that do not intersect~$\partial M\setminus E$.
%\item $M$ is a Euclidean domain and $E$ is strictly convex and open.
\end{enumerate}
If the broken ray transform of a function $f\in L^2(M)$ vanishes, then~$f=0$ in some neighborhood of~$E$. Furthermore, in the case~(2) $f$ vanishes on each geodesic with endpoints in~$E$. %in the convex hull of each connected component of $E$.
\end{theorem}

\begin{proof}
(1) Fix any $p\in E$ and use Theorem~\ref{thm:UV-local} near it.
The set~$O$ in Theorem~\ref{thm:UV-local} is constructed so that it can be shrinked to be inside any given neighborhood of~$p$ in~$M$.
(For details, see~\cite{UV:local-x-ray}.)
Thus in the present case we may choose so that $O\cap\partial M\subset E$.

If the broken ray transform of a function $f\in L^2(M)$ vanishes, its integral over any geodesic with endpoints in~$E$ is zero.
Therefore~$f$ vanishes in $O$ by Theorem~\ref{thm:UV-local}. The conclusion holds for all $p\in E$, whence~$f$ vanishes in a neighborhood of~$E$.

(2)
Follows from Theorem~\ref{thm:K-spt}.
\end{proof}

\begin{remark}
Let $U$ be the neighborhood of~$E$ in which~$f$ vanishes in part~(1) of the above theorem.
If $\partial U$ (in the topology of~$M$) is strictly convex from the side of $M\setminus U$, the theorem may be used on it again.
Such layer stripping might show that~$f$ vanishes in a relatively large neighborhood (like in part~(2)), whose geometry can be controlled more strongly.
The global injectivity result of Uhlmann and Vasy~\cite{UV:local-x-ray} is based on such an argument.
\end{remark}

\begin{remark}
\label{rmk:bdd-spt}
If we know that the support a function $f\in L^2(M)$ is a positive distance away from (the connected set)~$\bar{E}$ in case~(2) of the above theorem, we do not need to extend~$M$.
First, we replace $f$ by zero outside the convex hull of~$E$; this does not alter the integral of $f$ over geodesics with endpoints in~$E$, but makes sure that~$f$ has compact support in~$\sisus M$.
If now $\eps=d(\partial M,\spt f)>0$, we can use part~(2) of the above theorem on the manifold $M^\eps=\{x\in M:d(x,\partial M)\geq\eps\}$.
\end{remark}

\begin{remark}
In the Euclidean case one can simply use Helgason's support theorem.
This support theorem can be viewed as a special case of part~(2) of the above theorem.
\end{remark}

\section{Examples and counterexamples}
\label{sec:ex}

We list below some examples using Theorems~\ref{thm:brt2} and~\ref{thm:brt3} (which in turn are based on Theorem~\ref{thm:rfl}) and some counterexamples.

\begin{proposition}[Examples]
\label{prop:ex-thm}
The broken ray transform is injective in the following manifolds (with or without corners):
%The broken ray transform is injective in the class $L^2(A)$ for the following manifolds $A$ with corners:
\begin{enumerate}
\item Consider the quadrant of a sphere
\begin{equation}
%\label{eq:}
A=\{x\in S^n\subset\R^{n+1}:x_1\geq0,x_2\geq0\}
\end{equation}
when $n\geq3$ and the set of tomography $E=\{x\in A:x_1=0,x_2>0\}$. The broken ray transform is injective in the class $\{f\in L^2(A):\spt f\cap\bar{E}=\emptyset\}$.
\item The previous example with $n=2$ in the class $\{f\in C^\infty(A):\spt f\cap\bar{E}=\emptyset\text{ and }f\text{ is }\infty\text{-even at }\{x_2=0\}\}$.
\item The hemisphere
\begin{equation}
%\label{eq:}
M=\{x\in S^2\subset\R^3:x_2\geq0\}
\end{equation}
with the set of tomography $E=\{x\in\partial M:x_1<\eps\}$ for some $\eps>0$ in the class $\{f\in C^\infty(M):f\text{ is }\infty\text{-even at }\{x_2=0\}\}$.
\end{enumerate}
\end{proposition}
\begin{proof}
(1) We use Theorem~\ref{thm:brt3}.
Suppose~$f$ in the given class has vanishing broken ray transform.
Define $A^\eps=\{x\in A:x_1\geq\eps\}$ and let~$\eps$ be so small that $\spt f\subset A^\eps$.
Now~$f$ has vanishing broken ray transform in~$A^\eps$ with the set of tomography $E^\eps=\{x\in A:x_1=\eps,x_2>0\}$.

Let $M=\{x\in S^n\subset\R^{n+1}:x_2\geq0\}$ and $\rho(x)=C-d_M(x,e_1)^2$, where~$d_M$ is the intrinsic (Riemannian) metric on~$M$, $e_1=(1,0,\dots,0)\in A$, and~$C$ is a constant chosen so that $\rho(x)=0$ whenever $x\in E^\eps$.
By Theorem~\ref{thm:brt3} the broken ray transform is injective on~$A^\eps$, whence~$f=0$.

(2) This is the same as part~(1), only with Theorem~\ref{thm:brt2} instead of~\ref{thm:brt3}.

(3) Let $f$ be an unknown function in the given class with vanishing broken ray transform.
By Theorem~\ref{thm:brt-spt} (used in the sense of Remark~\ref{rmk:bdd-spt})~$f$ vanishes in $\{x_1\leq\eps\}$ which lies in the union of geodesics with endpoints in~$E$.
By part~(2) above~$f$ vanishes in $\{x_1\geq\eps\}$, too.
\end{proof}

\begin{proposition}[Counterexamples]
\label{prop:ctr-ex}
The broken ray transform fails to be injective on the following kinds of manifolds~$M$ and sets of tomography~$E$:
\begin{enumerate}
\item The manifold~$M$ is such that the geodesic ray transform is not injective, e.g. a one dimensional manifold.
$E\subset\partial M$ may be anything.
\item The manifold $M$ contains a reflecting tubular part: for $\Omega\subset\R^n$, $n\geq1$, a bounded $C^1$ set and $L>0$, the manifold with boundary $N=\bar{\Omega}\times(0,L)$ embeds isometrically to $M$ such that $\partial N=\partial\Omega\times(0,L)$ is mapped to the complement of~$E$.
\item The manifold $M$ contains a reflecting generalized tubular part: for $N_1$ and $N_2$ manifolds with boundary such that the geodesic ray transform on~$N_2$ is not injective, the manifold $N_1\times N_2$ embeds isometrically to~$M$ such that $\partial N_1\times N_2$ is mapped to the complement of~$E$.
We must have $\partial N_2\neq\emptyset$ but can have $\partial N_1=\emptyset$.
\end{enumerate}
Parts~(1) and~(3) hold for the function classes where the geodesic ray transform is non-injective.
Part~(2) holds for the class~$C_0^\infty$.
\end{proposition}

\begin{proof}
(1) In the case $E=\partial M$ the broken ray transform is the geodesic ray transform.
If the broken ray transform with set of tomography~$E$ is not injective, it is not injective with any set of tomography $E'\subset E$ either.

(2) Take a function $g:(0,L)\to\R$ such that $\int_0^Lg(t)\der t=0$ but~$g$ does not vanish identically.
%(For example $g(s)=\sin(2\pi s/L)$ or a variant with smooth continuation by zero.)
Define $f:N\to\R$ by $f(x,t)=g(t)$.
Using the embedding, extend~$f$ by zero to~$M$.
We claim that the broken ray transform of~$f$ vanishes.

It suffices to show that~$f$ integrates to zero on any (unit speed) broken ray~$\gamma$ in~$N$ starting at $\bar{\Omega}\times\{0\}$ and ending at $\bar{\Omega}\times\{L\}$.
Let $v=(0,\dots,0,1)\in\R^{n+1}$ be the unit vector normal to the hypersurfaces $\Omega\times\{s\}$.
Possible reflections at $\partial\Omega\times[0,L]$ are such that $\gamma'\cdot v$ is preserved.
Thus the integral over the broken ray becomes (up to a multiplicative constant) the integral of~$g$ over $(0,L)$, which vanishes.

(3) There is a function~$g$ in~$N_2$ such that it integrates to zero over all maximal geodesics in~$N_2$.
Define $f:N_1\times N_2\to\R$ by $f(x_1,x_2)=g(x_2)$.
For a unit speed broken ray in $N_1\times N_2$ with both endpoints in $N_1\times\partial N_2$ the $N_2$ component of the gradient is conserved in reflections and along geodesics.
Thus $f$ integrates to zero over any such broken rays just like in part~(2).
\end{proof}

\begin{remark}
Part~(2) of the above proposition is related to the fact that the geodesic ray transform on a one dimensional manifold is not injective; a function on the real line cannot be recovered from its integral.
Part~(3) naturally generalizes this observation.
\end{remark}

As an example of part~(3) with $\partial N_1=\emptyset$ we mention $N_1\times N_2=S^n\times(0,L)$.

There is a counterexample to the counterexample given in Proposition~\ref{prop:ctr-ex} which warns us that some counterexamples may fail when attenuation is introduced.
We give this as the following proposition.
The result could be given for more general manifolds and broken rays, but we only state it here for the simple cylindrical case.

\begin{proposition}
\label{prop:ex-att}
Consider the manifold $M=[0,L]\times S^1$ with boundary. Let $a\geq0$ be a constant attenuation coefficient.
For a function $g\in C([0,L])$ define $f_g:M\to\R$ by $f_g(x,y)=g(x)$.
\begin{enumerate}
\item If $a=0$, there is a nonzero function $g\in C([0,L])$ such that the ray transform $\rt{a}f_g$ vanishes.
\item If $a>0$, there is no nonzero function $g\in C([0,L])$ such that the ray transform $\rt{a}f_g$ vanishes.
\end{enumerate}
\end{proposition}
\begin{proof}
(1) Choose a smooth~$g$ which integrates to zero as in the proof of Proposition~\ref{prop:ctr-ex}(2).

(2) Consider geodesics from $\{0\}\times S^1$ to $\{L\}\times S^1$ of the form $\gamma_b:[0,L/b]\to M$, $\gamma_b(t)=(bt,\exp(i\sqrt{1-b^2}t))$, where $b\in(0,1]$.
(All geodesics are of this form up to trivial transformations.)
We wish to show that if $\rt{a}f_g(\gamma_b)=0$ for all $b\in(0,1]$, then~$g=0$.

After extending~$g$ by zero to $[0,\infty)$ we find
\begin{equation}
%\label{eq:}
\begin{split}
\rt{a}f_g(\gamma_b)
&=
\int_0^{L/b}e^{-at}f_g(\gamma_b(t))\der t
\\&=
\int_0^\infty e^{-at}g(bt)\der t
\\&=
b^{-1}\lap g(a/b),
\end{split}
\end{equation}
where $\lap g(s)=\int_0^\infty e^{-st}g(t)\der t$ is the Laplace transform of~$g$.
Thus, if $\rt{a}f(\gamma_b)=0$ for all $b\in(0,1]$, we have that $\lap g(s)=0$ for all $s\in[a,\infty)$.

Since~$g$ is bounded and has compact support,~$\lap g$ is real analytic on $(0,\infty)$.
But~$\lap g$ vanishes in $[a,\infty)$, so~$\lap g=0$.
It follows from the properties of the Laplace transform that~$g=0$.
\end{proof}

\section{The periodic broken ray transform}
\label{sec:periodic}

In analogue to the broken ray transform introduced in the beginning of Section~\ref{sec:intro}, we now turn to the periodic broken ray transform.
In this case the entire boundary~$\partial M$ is reflecting and the integrals of the unknown function are known over periodic broken rays.

Periodic broken rays are analogous to periodic geodesics on a closed manifold, and this analogy is made precise in the proof of the following two propositions.
Guillemin and Kazhdan~\cite{GK:spectral2D} reduced spectral rigidity of negatively curved closed Riemannian surfaces to determining a function from its integrals over all periodic geodesics.
We therefore expect spectral rigidity of negatively curved surfaces with boundary to be related to determining a function from its integrals over all periodic broken rays.
Lengths of periodic broken rays (or periodic billiard orbits) play an important role in spectral geometry (see~\cite{DH:spectral-survey}).
Since linearizing lengths of geodesics with respect to the metric leads to X-ray transforms, the periodic broken ray transform can be expected to have applications in spectral geometry.

In the introductory examples of Section~\ref{sec:ex1} and more generally in Theorem~\ref{thm:rfl} the injectivity of broken ray transforms was reduced to injectivity of certain related geodesic ray transforms via reflections.
The geodesics and broken rays considered there joined two points on the boundary or the set of tomography.

The same idea can be carried over to the case of the periodic broken ray transform.
We study this idea briefly in this section.
The periodic broken ray transform were defined in definition~\ref{def:pbrt}.

%xxx Old position of definition \label{def:rt-brt}

It is clear that the periodic broken ray transform fails to be injective if there are too few periodic broken rays on the manifold.
We consider below specific examples, where the geometry allows for a large number of periodic broken rays.

\begin{proposition}
\label{prop:ex-S/8}
The periodic broken ray transform is injective for the Riemannian manifold with boundary $M=\{(x_1,x_2,x_3)\in S^{2}:x_i\geq0\forall i\}$ when restricted to smooth functions with vanishing normal derivatives of odd order at the boundary.
\end{proposition}
\begin{proof}
Let $f:M\to\R$ be a smooth function with vanishing normal derivatives of odd order at the boundary such that $\brt{} f=0$.
We need to show that~$f=0$.

Define a map $p:S^2\to M$ such that $p(x,y,z)=(\abs{x},\abs{y},\abs{z})$.
Let $\rfl{f}:S^2\to\R$ be defined simply by $\rfl{f}=f\circ p$.
It follows from the assumptions that~$\rfl{f}$ is smooth.

If~$\phi$ is a (closed) geodesic in~$S^2$ (a great circle), then $p\circ\phi\in\Gamma$; in fact, if~$\rfl{\Gamma}$ is the set of geodesics in~$S^2$, then $p\circ\rfl{\Gamma}=\Gamma$.
(Note that this argument relies on the geometry of~$S^2$. In particular, all geodesics are closed.)

Let~$\rt{}$ be the geodesic ray transform in~$S^2$.
For any $\gamma\in\Gamma$ the set $(p\circ{})^{-1}(\gamma)=\{\phi\in\rfl{\Gamma}:p\circ\phi=\gamma\}$ is nonempty and finite and $\rt{}\rfl{f}(\phi)=\brt{} f(\gamma)$ for all $\phi\in (p\circ{})^{-1}(\gamma)$.
Thus
\begin{equation}
%\label{eq:}
\rt{}\rfl{f}(\phi)=0\text{ for all }\phi\in\rfl{\Gamma}
\Leftrightarrow
\brt{} f(\gamma)=0\text{ for all }\gamma\in\Gamma.
\end{equation}
Therefore the assumption $\brt{} f=0$ implies that~$\rt{}\rfl{f}=0$.

If the geodesic ray transform of a continuous function $\rfl{f}:S^2\to\R$ vanishes, then~$\rfl{f}$ is odd in the sense that $\rfl{f}(-x)=-f(x)$ for all $x\in S^2$ by~\cite[Theorem~1.13 on page~9]{PM:ipgd-notes}.
By construction, however, $\rfl{f}$ is also even, so~$\rfl{f}=0$.
Therefore~$f=0$.
\end{proof}

\begin{proposition}
\label{prop:pbrt-cube}
The periodic broken ray transform is injective in the unit cube $[0,1]^n\subset\R^n$, $n\geq2$, in the class of functions $\{f\in C^n([0,1]^n):f$ is $n$-odd at the boundary$\}$.
\end{proposition}

%This result relies on a result in the $n$-torus by Abouelaz and Rouvi\`ere~\cite{AR:radon-torus}, and the inversion formula therein immediately gives an inversion formula for the periodic broken ray transform in the cube.
This result relies on a result by Abouelaz and Rouvi\`ere~\cite{AR:radon-torus} in the $n$-torus, and the inversion formula therein immediately gives an inversion formula for the periodic broken ray transform in the cube.
Similarly, the range characterization of~\cite{AR:radon-torus} can be turned into a characterization of the range of the periodic broken ray transform in the cube.

\begin{proof}[Proof of Proposition~\ref{prop:pbrt-cube}]
We define a map $p:[0,2]^n\to[0,1]^n$ by
\begin{equation}
%\label{eq:}
p(x_1,\dots,x_n)
=
(1-\abs{1-x_1},\dots,1-\abs{1-x_n}).
\end{equation}
We make $[0,2]^n$ into a flat $n$-torus by identifying opposite faces of the cube.

Let $f$ be a function in the class of the claim.
The regularity assumption implies that $\rfl{f}=p^*f\in C^n([0,2]^n)$.
The torus is obtained by gluing together reflected copies of the original cube in a natural way, and~$\rfl{f}$ is the corresponding reflection of the function~$f$.

Using the reflection argument used in the proof of Propositions~\ref{prop:ex-cone} and~\ref{prop:ex-S/8} it is easy to observe that if~$\gamma$ is a periodic broken ray in the cube $[0,1]^n$, then each $\rfl{\gamma}\in (p\circ{})^{-1}(\gamma)$ is a closed geodesic in the torus $[0,2]^n=\R^n/(2\Z)^n$.
The integral of $f$ over $\gamma$ yields integrals of~$\rfl{f}$ over~$\rfl{\gamma}$ as in~\eqref{eq:rfl-int}.
Therefore the periodic broken ray transform of~$f$ determines the integral of~$\rfl{f}$ over all closed geodesics of the torus.

Since this information is enough to determine to determine~$\rfl{f}$ (see~\cite{AR:radon-torus}) and $f=\rfl{f}\circ p$, also~$f$ may be recovered.
\end{proof}

We wish to point out the similarity between Proposition~\ref{prop:ex-S/8} and Proposition~\ref{prop:ex-thm}(3).
Similarly Proposition~\ref{prop:pbrt-cube} should be compared with Proposition~\ref{prop:ex-cone}(1), which contains the square (with two adjacent edges as~$E$) as a special case.
The geometrical construction is very similar, but the underlying result for the ray transform on~$S^2$ is quite different.

In addition to the examples of injective periodic broken ray transforms in Propositons~\ref{prop:ex-S/8} and~\ref{prop:pbrt-cube} we also give a counterexample.

\begin{proposition}
\label{prop:ex-disk}
There exists a compactly supported nonvanishing smooth function in the unit disk such that its periodic broken ray transform vanishes.
\end{proposition}

\begin{proof}
Let $g:(0,1)\to\R$ be a nonzero smooth function with compact support and define $f(r,\theta)=g(r)\cos(\theta)$ in polar coordinates.
Now~$f$ is smooth in the unit disk, and we wish to show that its integral vanishes over every periodic broken ray.

For a fixed broken ray, rotate the coordinates so that one point of reflection is at angle zero.
If the needed rotation angle is~$\phi$, we have in the new polar coordinates $(r,\theta')=(r,\theta+\phi)$
\begin{equation}
%\label{eq:}
f(r,\theta')
=
g(r)
[\cos(\theta')\cos(\phi)-\sin(\theta')\sin(\phi)].
\end{equation}
The second term is antisymmetric with respect to the reflection $\theta'\mapsto-\theta'$ but the broken ray is symmetric.
Thus the integral of the second term vanishes over the broken ray.
The remaining term has (apart from the constant $\cos(\phi)$) the same form as the original function $f(r,\theta)=g(r)\cos(\theta)$.
It therefore suffices to show that the integral of~$f$ vanishes over any broken ray with one reflection at angle zero.

By~\cite[corollary~13]{I:disk} the integral of~$f$ over any such broken ray with two or more reflections vanishes.
In a disk any periodic broken ray has at least two reflections.
\end{proof}

Some information can, however, be recovered from the periodic broken ray transform of a continuous function in the disk.
Although Proposition~\ref{prop:ex-disk} prohibits full reconstruction, we can construct the function at the origin and its integral over any circle centered at the origin~\cite[theorem~1]{I:disk}.

\section*{Acknowledgements}
The author is partly supported by the Academy of Finland (no 250~215).
The author wishes to thank Mikko Salo for discussions regarding this article and the referee for useful feedback.
Part of the work was done during a visit to the Institut Mittag-Leffler (Djursholm, Sweden).

\bibliographystyle{mscplain}%mscplain
\bibliography{manifold-refl-arxiv}

\end{document}